\documentclass[a4paper,12pt,reqno]{amsart}
\usepackage{latexsym}
\usepackage{amssymb,amsfonts,amsmath,mathrsfs}
\usepackage{geometry}
\usepackage{verbatim}
\newcommand{\eps}{\varepsilon}
\newcommand{\pr}[1]{\left( #1\right)}

\let\originalleft\left
\let\originalright\right
\renewcommand{\left}{\mathopen{}\mathclose\bgroup\originalleft}
\renewcommand{\right}{\aftergroup\egroup\originalright}
\newcommand{\es}[1]{\begin{equation}\begin{split}#1\end{split}\end{equation}}
\newcommand{\est}[1]{\begin{equation*}\begin{split}#1\end{split}\end{equation*}}
\newcommand{\R}{\mathbb{R}}

\newcommand{\C}{\mathbb{C}}

\newcommand{\Z}{\mathbb{Z}}
\renewcommand{\mod}[1]{~\pr{\textnormal{mod}~#1}}
\newcommand{\e}[1]{\operatorname{e}\pr{ #1}}

\newcommand{\tRe}{\textup{Re}}

\newtheorem{theorem}{Theorem}[section]
\newtheorem{proposition}{Proposition}[section]
\newtheorem{corollary}{Corollary}[section]
\newtheorem{lemma}{Lemma}[section]

\newtheorem*{remark*}{Remark}

\begin{document}

\allowdisplaybreaks

\title[Moments of the Riemann zeta-function]{A quadratic divisor problem and moments of the Riemann zeta-function}
\author{Sandro Bettin, H. M. Bui, Xiannan Li and Maksym Radziwi\l\l}
\address{Dipartimento di Matematica, Via Dodecaneso 35, 16146 Genova, ITALY}
\email{bettin@dima.unige.it}
\address{School of Mathematics, University of Manchester, Manchester M13 9PL, UK}
\email{hung.bui@manchester.ac.uk}
\address{Mathematics Department,138 Cardwell Hall, Manhattan, KS, 66506, USA}
\email{xiannan@ksu.edu}
\address{Department of Mathematics, McGill University, 805 Sherbrooke West, Montreal, Quebec, H3A 0B9, Canada}
\email{maksym.radziwill@gmail.com}

\begin{abstract}
We estimate asymptotically the fourth moment of the Riemann zeta-function twisted by a Dirichlet polynomial of length $T^{\frac14 - \varepsilon}$. Our work relies crucially on Watt's theorem on averages of Kloosterman fractions. In the context of the twisted fourth moment, Watt's result is an optimal replacement for Selberg's eigenvalue conjecture.
 
Our work extends the previous result of Hughes and Young, where Dirichlet polynomials of length $T^{\frac{1}{11}-\varepsilon}$ were considered. 
Our result has several applications, among others to the proportion of critical zeros of the Riemann zeta-function, zero spacing and lower bounds for moments. 

Along the way we obtain an asymptotic formula for a quadratic divisor problem, where the condition $a m_1 m_2 - b n_1 n_2 = h$ 
is summed with smooth averaging  on the variables
$m_1, m_2, n_1, n_2, h$ and arbitrary weights in the average on $a,b$. Using Watt's work allows us to exploit all averages
simultaneously. It turns out that averaging over $m_1, m_2, n_1, n_2, h$ right away in the quadratic divisor problem simplifies considerably the combinatorics of the main terms in the twisted fourth moment.  
\end{abstract}

\maketitle
\section{Introduction}
\begin{comment}
The quadratic divisor problem concerns the estimation of
$$
\sum_{\substack{a n - b m = h \\ X \leq n \leq 2X}} d(n) d(m).
$$
The problem first arose in Ingham's classical work on the fourth moment of the Riemann zeta-function. 
Since then variants of this problem have revealed themselves to be central in many investigations, notably
in the proof of subconvexity bounds or in the estimation of moments in other families. 

Estimates for the quadratic divisor problem are intimately related to 
\end{comment}

The Riemann zeta-function $\zeta(s)$ is intimately related to the study of prime numbers and other problems in number theory.  There are a number of famous conjectures in this area. Two distinguished examples are the Riemann Hypothesis, which states that all non-trivial zeros of $\zeta(s)$ are on the line $\tRe (s) = 1/2$, and the Lindel\"of Hypothesis, which states that $\zeta(1/2+it)\ll_\varepsilon (1+|t|)^\varepsilon$.

These two conjectures remain far out of reach.  However, methods in analytic number theory can prove that these conjectures are true on average.  An example of this is the study of moments of $\zeta(s)$.  To be more precise, let
$$I_k(T) = \int_0^T |\zeta(\tfrac{1}{2} + it)|^{2k}dt.
$$  Here, asymptotic formulae were proven for $k = 1$ by Hardy and Littlewood and for $k=2$ by Ingham (see [\ref{Ti}; Chapter VII]).  Note that the Lindel\"of Hypothesis is equivalent to $I_k(T) \ll_\eps T^{1+\eps}$ for all $k\in\mathbb{N}$.

The result of Ingham was useful in proving his zero density result (see, for example, \cite{Ti}), which also has applications to prime numbers.  Despite extensive further work, no such result is available for any other values of $k$.  However, results are available for twisted fourth moments of $\zeta(s)$, which may be considered to be somewhere between the $k=2$ result of Ingham and the open problem for $k=3$.  Let us define
\begin{equation*}
P(s)=\sum_{a\leq T^\vartheta}\frac{\alpha_a}{a^{s}}
\end{equation*} to be a Dirichlet polynomial of length $T^\vartheta$, with $\vartheta\geq 0$ and $\alpha_a\ll a^\eps$.  Then Watt's result in \cite{Watt} gives that
\begin{equation}\label{eqn:twistedmoment1}
\int_0^T|\zeta(\tfrac{1}{2}+it)|^4 |P(\tfrac{1}{2}+it)|^2 dt \ll_\eps T^{1+\varepsilon} 
\end{equation} for $\vartheta < 1/4$.  This is an improvement over the work of Deshouillers and Iwaniec \cite{DI}, which had a similar bound for $\vartheta < 1/5$, and the initial work of Iwaniec \cite{I}, which led to $\vartheta < 1/10$ just using the Weil bound. Despite appearances, this type of bound is not far removed from the prime number theory which inspired such questions.  For instance, the bound~\eqref{eqn:twistedmoment1} is useful in studying prime numbers in short intervals \cite{BHP}.

It is desirable to evaluate more precisely the quantity in~\eqref{eqn:twistedmoment1}, in view of various applications to the theory of the Riemann zeta-function, including the study of proportion of zeros on the critical line, gaps between zeros of the zeta-functions, and lower bounds for moments. Some of these consequences of our main results below have been in fact already worked out (see \cite{B, BHT, BM}) and have remained thus far conditional. 

Hughes and Young \cite{HY} obtained an asymptotic formula for 
\begin{equation*}
\int_0^T|\zeta(\tfrac{1}{2}+it)|^4 |P(\tfrac{1}{2}+it)|^2 dt
\end{equation*}when $\vartheta < 1/11$, and it is expected that this result remains true all the way for $\vartheta < 1$ (and in this range it implies the Lindel\"of Hypothesis).  In this paper, we prove the following.
\begin{theorem}\label{mt}
Let $T\geq2$ and let $\alpha,\beta,\gamma,\delta\in\C$ with $\alpha,\beta,\gamma,\delta\ll(\log T)^{-1}$. Furthermore, let $\Phi(x)$ be a smooth function supported in $[1, 2]$ with derivatives $\Phi^{(j)}(x) \ll_j T^\epsilon$ for any $j\geq0$.  Consider 
$$
A(s)= \sum_{a \leq T^{\vartheta}} \frac{\alpha_a}{a^s}\quad \text{ and }\quad B(s) = \sum_{b \leq T^{\vartheta}} \frac{\beta_b}{b^s},
$$
where $\alpha_a \ll a^{\varepsilon}$ and $\beta_b \ll b^{\varepsilon}$,
and let $I_{\alpha, \beta, \gamma, \delta}(T)$ denote
\est{
\int_\R \zeta\pr{\tfrac12+it+\alpha}\zeta\pr{\tfrac12+it+\beta}\zeta\pr{\tfrac12-it+\gamma}\zeta\pr{\tfrac12-it+\delta} A(\tfrac 12 + it) \overline{B(\tfrac 12 + it)}
\Phi\Big(\frac tT\Big)\,dt.
}
Define
\[
Z_{\alpha,\beta,\gamma,\delta,a,b}=A_{\alpha,\beta,\gamma,\delta}B_{\alpha,\beta,\gamma,\delta,a}B_{\gamma,\delta,\alpha,\beta,b},
\]
where
\begin{align*}
A_{\alpha, \beta, \gamma, \delta} =&   
\frac{\zeta(1 + \alpha + \gamma) \zeta(1 +\alpha + \delta) \zeta(1 + \beta + \gamma)\zeta(1+\beta+\delta)}{\zeta(2 + \alpha + \beta + \gamma + \delta)}, 
\end{align*}
\est{
B_{\alpha,\beta,\gamma,\delta,a}=\prod_{p^\nu||a}\left(\frac{\sum_{j=0}^{\infty}\sigma_{\alpha,\beta}(p^j)\sigma_{\gamma,\delta}(p^{j+\nu})p^{-j}}{\sum_{j=0}^{\infty}\sigma_{\alpha,\beta}(p^j)\sigma_{\gamma,\delta}(p^{j})p^{-j}}\right)
}
and $\sigma_{\alpha,\beta}(n)=\sum_{n_1n_2=n}n_{1}^{-\alpha}n_{2}^{-\beta}$. Then we have 
\begin{align*}
& I_{\alpha,\beta,\gamma,\delta}(T)=  \sum_{g} \sum_{(a,b) = 1} \frac{\alpha_{ga} \overline{\beta_{gb}}}{g ab}\int_\R\Phi\Big(\frac tT\Big)\bigg ( Z_{\alpha, \beta, \gamma, \delta,a,b} + \Big ( \frac{t}{2\pi} \Big )^{-\alpha - \beta - \gamma - \delta} Z_{-\gamma, -\delta, -\alpha, -\beta,a,b} \bigg ) \,dt\\ &\qquad\quad + 
\sum_{g} \sum_{(a,b) = 1} \frac{\alpha_{ga} \overline{\beta_{gb}}}{g ab}\int_\R\Phi\Big(\frac tT\Big)\bigg ( \Big ( \frac{t}{2\pi} \Big )^{-\alpha - \gamma} Z_{-\gamma, \beta,-\alpha, \delta,a,b}
+ \Big ( \frac{t}{2\pi} \Big )^{-\alpha - \delta} Z_{-\delta, \beta, \gamma, -\alpha,a,b} \\ & \qquad\qquad\qquad\qquad
  + \Big ( \frac{t}{2\pi} \Big )^{-\beta - \gamma} Z_{\alpha, -\gamma, -\beta, \delta,a,b}+ \Big ( \frac{t}{2\pi} \Big )^{-\beta - \delta} Z_{\alpha, -\delta, \gamma, -\beta,a,b} \bigg)\,dt\\
&\qquad\qquad\qquad\qquad\qquad\qquad+O_\eps\Big(T^{\frac12+2\vartheta + \varepsilon}+T^{\frac34+\vartheta+\varepsilon}\Big).
\end{align*}
\end{theorem}

\noindent\textbf{Remarks.}\begin{itemize}
\item Setting $A = B$ and letting the shifts $\alpha, \beta, \gamma, \delta \rightarrow 0$, Theorem~\ref{mt} implies an asymptotic formula for 
\begin{equation*}
\int_0^T |\zeta(\tfrac12+it)|^4 |P(\tfrac12+it)|^2 dt
\end{equation*}when $\vartheta < 1/4$, which should be compared to the $\vartheta <1/11$ restriction in the work of Hughes and Young \cite{HY}.
\item The above expression coincides with that obtained by Hughes and Young \cite{HY}. Here, the first two terms come from the diagonal, while the four remaining terms are the main terms coming from the off-diagonal contribution of sums of the following type
\begin{equation*}% \label{equ:quad_div}
\sum_{a m_1 m_2 - b n_1 n_2 = h \neq 0} \frac{\alpha_a \overline{\beta_b}}{m_1^{\alpha} m_2^{\beta} n_1^{\gamma} n_2^{\delta}}  f(a m_1 m_2, b n_1 n_2, h) K(m_1 m_2 n_1 n_2).
\end{equation*}
Each of the four possibilities where $n_1 < n_2$ or $n_1>  n_2$, $m_1 < m_2$ or $m_1> m_2$ contributes to exactly one of the off-diagonal main terms.
\item As mentioned in [\ref{HY}; page 207], the symmetries of the expression imply that the sum of the six main terms is holomorphic in terms of the shift parameters. The holomorphy of this permutation sum has been proved in [\ref{CFKRS}; Lemma 2.5.1]. In the remaining of the article, we impose the additional restrictions that $|\alpha\pm\beta|\gg(\log T)^{-1}$, etc. We note that the holomorphy of $I_{\alpha,\beta,\gamma,\delta}(T)$ and of the permutation sum leads to the holomorphy of the error term, and hence the maximum modulus principle can be applied to extend the error term to the enlarged domain.
\end{itemize}

Practically, it is however unnecessary to specify the Euler products $A_{\alpha, \beta, \gamma, \delta}$ and $B_{\alpha,\beta,\gamma,\delta,a}$. In various applications (for example, \cite{B, BHT, BM}), the resulting arithmetic factor can be worked out much more easily by incorporating the arithmetic properties of the sequences $\alpha_a$ and $\beta_b$. For that purpose we state a variant of Theorem~\ref{mt} below.

\begin{theorem}\label{variantmt}
Under the same assumptions as in Theorem~\ref{mt} we have
\begin{align*}
& I_{\alpha,\beta,\gamma,\delta}(T)= \sum_{a,b\leq T^\vartheta} \alpha_{a} \overline{\beta_{b}}\int_\R \Phi\Big(\frac tT\Big)\bigg (\widetilde{Z}_{\alpha, \beta, \gamma, \delta,a,b}(t) + \Big ( \frac{t}{2\pi} \Big )^{-\alpha - \beta - \gamma - \delta} \widetilde{Z}_{-\gamma, -\delta, -\alpha, -\beta,a,b}(t) \bigg )dt \\ &\qquad\quad + 
 \sum_{a,b\leq T^\vartheta}\alpha_{a} \overline{\beta_{b}}\int_\R \Phi\Big(\frac tT\Big)\bigg ( \Big ( \frac{t}{2\pi} \Big )^{-\alpha - \gamma} \widetilde{Z}_{-\gamma, \beta,-\alpha, \delta,a,b}(t)
+ \Big ( \frac{t}{2\pi} \Big )^{-\alpha - \delta} \widetilde{Z}_{-\delta, \beta, \gamma, -\alpha,a,b}(t) \\ & \qquad\qquad\qquad\qquad
  + \Big ( \frac{t}{2\pi} \Big )^{-\beta - \gamma} \widetilde{Z}_{\alpha, -\gamma, -\beta, \delta,a,b}(t)+ \Big ( \frac{t}{2\pi} \Big )^{-\beta - \delta} \widetilde{Z}_{\alpha, -\delta, \gamma, -\beta,a,b}(t) \bigg)dt\\
&\qquad\qquad\qquad\qquad\qquad\qquad+O_\eps\Big(T^{\frac12+2\vartheta + \varepsilon}+T^{\frac34+\vartheta+\varepsilon}\Big),
\end{align*}
where
\begin{align*}
\widetilde{Z}_{\alpha, \beta, \gamma, \delta,a,b}(t)=\sum_{am_1m_2=bn_1n_2}\frac{1}{(ab)^{\frac12}m_{1}^{\frac12+\alpha}m_{2}^{\frac12+\beta}n_{1}^{\frac12+\gamma}n_{2}^{\frac12+\delta}}V^*\Big(\frac{m_1m_2n_1n_2}{t^2}\Big)
\end{align*}
and the function $V^*(x)$ is defined as in~\eqref{aeqac}.
\end{theorem}

\noindent\textbf{Remark.}\ \ Note that the function $V^*(x)$ satisfies $V^*(x)\ll_A(1+|x|)^{-A}$ for any fixed $A>0$, so Theorem~\ref{variantmt} shows a better structure of the main terms. This is the form suggested by following the recipe in \cite{CFKRS}.

 An important feature of our results is that we exploit the averaging over $a, b$ in the proof of the theorems. Thus stating the results for individual $a,b$ and then summing the error term would lead to an inferior bound. Another interesting feature is that since we arrive to the main terms from another direction, the combinatorics of the main terms turn out to be easier than in previous treatments.

Our results should also be contrasted with recent results in \cite{BCR}, where the length of $\vartheta$ was extended beyond $1/2$ for the twisted second moment, and where some expressions approaching those of Theorem~\ref{mt} were considered. In addition, the range $\vartheta < 1/4$ is optimal in the sense that assuming the Selberg eigenvalue conjecture does not lead to an extension of the range of $\vartheta$.  On the Selberg eigenvalue conjecture Motohashi \cite{M} has obtained an exact formula for the twisted fourth moment. However in his treatment an estimation of the error terms is lacking (and the average over $a$ and $b$ is not exploited), and should not in any case allow one to exceed $\vartheta =1/4$, as we will now explain. 
If the polynomial is chosen to be an amplifier of length $T^{\frac14-\eps}$, then results of this form lead to the Burgess style subconvexity bound $|\zeta(1/2+it)|\ll_\eps t^{\frac{3}{16}+\eps}$.  Since this bound is a natural barrier in other families of $L$-functions, it seems likely that we cannot improve the length of the polynomial without including new ingredients specific to $\zeta(s)$.

The improvement over the work of Hughes and Young \cite{HY} arises from two ingredients, both appearing in the treatment of a shifted convolution problem involving the divisor function.  The first is that we do not use the $\delta$-method, which turns out to be suboptimal in this application.  The second, and main reason for the improvement in our work, is the treatment of an exponential sum, which resembles a sum of Kloosterman sums.  In Hughes and Young's work, they use the Weil bound for Kloosterman sums, neglecting the possibility of further cancellation in the sum.  Our work takes advantage of further cancellation derived from spectral theory on $GL(2)$.  In particular, we use the exponential sum bound from Watt \cite{Watt}, which is based on the work of Deshouillers and Iwaniec \cite{DI2}. However, we also appeal in certain circumstances to the Weil bound, when Watt's result is not effective. 

The quadratic divisor problem that we obtain is likely to be useful in other work, and therefore we also state it here. For a function $f(x,y, z)$ decaying sufficiently fast at infinity, we let $\widehat f_3(x, y, s)$ denote the Mellin transform of $f$ with respect to the third variable and we write $\widehat f$ for the Mellin transform with respect to all three variables. 
Further, let $\widetilde{f}_{ \alpha, \beta, \gamma, \delta}(x, y; a,b,g)$ be
\begin{align}\label{dfntildef}
\frac{1}{2\pi i} \int_{(1+\eps)}\widehat{f}_3(x, y, s)\zeta(s) \zeta(1 + \alpha - \beta+ \gamma - \delta  + s) g^{- s} \eta_{\alpha,\beta,\gamma,\delta,a,b}(0,0,s) 
ds,
\end{align}
where $ \eta_{\alpha,\beta,\gamma,\delta,a,b}(u,v,s)$ is defined as in~\eqref{etanew}. Then we have the following. 
\begin{theorem}\label{thmqdp}
Let $A,B,X,Z,T\geq 1$ with $Z>XT^{-\eps}$ and $\log(ABXZ)\ll \log T$. Let $\alpha_a,\beta_b$ be sequences of complex numbers supported on $[1,A]$ and $[1,B]$, respectively, and such that $\alpha_a\ll A^\eps, \beta_b\ll B^\eps$. Let $f\in\mathcal C^{\infty}(\R_{\geq0}^{3})$ be such that
\est{
\frac{\partial ^{i+j+k}}{\partial x^i\partial y^j\partial z^k}f(x,y,z)\ll_{i,j,k,r} T^{\eps}(1+x)^{-i}(1+y)^{-j}(1+z)^{-k}\Big(1+\frac{z^2Z^2}{xy}\Big)^{-r}
}
for any $i,j,k,r\geq0$. Let $K\in\mathcal C^{\infty}(\R_{\geq0})$ be such that $K^{(j)}(x)\ll_{j,r} T^\eps (1+x)^{-j}(1+x/X^2)^{-r}$ for  any $j,r\geq0$. Then, writing 
$$
\mathcal{S} = \sum_{a m_1 m_2 - b n_1 n_2 = h > 0} \frac{\alpha_a \overline{\beta_b}}{m_1^{\alpha} m_2^{\beta} n_1^{\gamma} n_2^{\delta}} 
 f(a m_1 m_2, b n_1 n_2, h) K(m_1 m_2 n_1 n_2),
$$
where the sum runs over positive integers $a,b,m_1,m_2,n_1,n_2$ and $h$,
we have
$$
\mathcal{S} = \mathcal{M}_{\alpha, \beta, \gamma, \delta} + \mathcal{M}_{\beta, \alpha, \gamma, \delta} + \mathcal{M}_{\alpha, \beta, \delta, \gamma} + \mathcal{M}_{\beta, \alpha, \delta, \gamma} + \mathcal{E},
$$
where
\begin{align*}
 \mathcal{M}_{\alpha,\beta,\gamma,\delta}=& \frac{\zeta(1 + \alpha -\beta)\zeta(1 + \gamma-\delta)}{\zeta(2 + \alpha - \beta + \gamma- \delta)}\sum_{g}
\sum_{(a,b) = 1} \frac{\alpha_{ga} {\beta_{gb}}g}{(g a )^{1 -\beta} (g b)^{1-\delta }} 
 \\ &\qquad\qquad \int_{0}^{\infty} K\Big(\frac{x^2}{g^2ab}\Big)\widetilde{f}_{ \alpha, \beta, \gamma, \delta}(x, x; a,b,g) x^{-\beta-\delta } dx
\end{align*}
and the error term $\mathcal{E}$ is bounded by
\begin{align*}
\mathcal E  \ll& T^{\eps} (AB)^\frac12 XZ^{-\frac12}  \Big ( AB + (A+ B)^{\frac12} (AB)^\frac14X^\frac12Z^{-\frac14} \Big ).
\end{align*}
\end{theorem}
Another variant is stated in Section 4. We have chosen to state in the introduction 
the version that we will use to obtain Theorem~\ref{mt}. 
Here, as explained before, each of the four main terms comes from the four possibilities where $n_1 < n_2$ or $n_1> n_2$,
$m_1 < m_2$ or $m_1 > m_2$. To contrast our result with previous work, the novelty in our treatment is that we average over all 
possible parameters, while allowing the averages over $a,b$ to have arbitrary weights. 
In comparison, the $\delta$-method delivers a fairly poor range of admissible values of $a,b$. % and does not depend on any spectral theory.   
Finally, when $a = b = 1$ strong error terms have been obtained by Motohashi \cite{M1} exploiting the fact that there are no
exceptional eigenvalue for the Laplacian on $SL(2, \mathbb Z)\backslash \mathcal H$, for $\mathcal H$ the usual upper half plane. 

\section{Proof of Theorem~\ref{mt} and Theorem~\ref{variantmt}}
\subsection{The approximate functional equation}
We start by recalling the approximate functional equation.
\begin{lemma}[Approximate functional equation] 
Let $G(s)$ be an even entire function of rapid decay in any fixed strip $|\emph{Re}(s)|\leq C$ satisfying $G(0)=1$, and let
\est{
V_{\alpha,\beta,\gamma,\delta}(x,t)=\frac1{2\pi i}\int_{(1)}\frac{G(s)}{s}g_{\alpha,\beta,\gamma,\delta}(s,t)\pi^{-2s}x^{-s}\,ds,
}
where 
\est{
g_{\alpha, \beta, \gamma, \delta}(s,t) = 
\frac{\Gamma\left(\frac{\frac12 + \alpha + s +it}{2} \right)}{\Gamma\left(\frac{\frac12 + \alpha +it }{2} \right)} 
\frac{\Gamma\left(\frac{\frac12 + \beta + s +it}{2} \right)}{\Gamma\left(\frac{\frac12 + \beta +it }{2} \right)} 
\frac{\Gamma\left(\frac{\frac12 + \gamma + s -it}{2} \right)}{\Gamma\left(\frac{\frac12 + \gamma -it }{2} \right)}
\frac{\Gamma\left(\frac{\frac12 + \delta +s -it }{2} \right)}{\Gamma\left(\frac{\frac12 + \delta -it }{2} \right)}.
}
Furthermore, set
\est{
X_{\alpha,\beta,\gamma,\delta}(t) = \pi^{\alpha + \beta + \gamma + \delta} 
\frac{\Gamma\left(\frac{\frac12 -\alpha - it}{2}\right)}{\Gamma\left(\frac{\frac12 + \alpha + it}{2}\right)}
\frac{\Gamma\left(\frac{\frac12 -\beta - it}{2}\right)}{\Gamma\left(\frac{\frac12 + \beta + it}{2}\right)}
\frac{\Gamma\left(\frac{\frac12 -\gamma + it}{2}\right)}{\Gamma\left(\frac{\frac12 + \gamma - it}{2}\right)}
\frac{\Gamma\left(\frac{\frac12 -\delta + it}{2}\right)}{\Gamma\left(\frac{\frac12 + \delta - it}{2}\right)}
}
and 
\est{
 \widetilde V_{\alpha,\beta,\gamma,\delta}(x,t)=X_{-\gamma,-\delta,-\alpha,-\beta}(t)V_{\alpha,\beta,\gamma,\delta}(x,t).
}
Then we have
\begin{align}\label{afeq}
&\zeta({\tfrac12 + \alpha + it}) \zeta({\tfrac12 + \beta + it}) \zeta({\tfrac12 + \gamma - it}) \zeta({\tfrac12 + \delta - it})
\nonumber\\
&\qquad=\sum_{m,n} \frac{\sigma_{\alpha,\beta}(m) \sigma_{\gamma,\delta}(n)}{(mn)^{\frac12 }} \left(\frac{m}{n}\right)^{-it} V_{\alpha, \beta, \gamma, \delta} \left(mn,t \right)
\\
&\qquad\qquad+ 
   \sum_{m,n} \frac{\sigma_{-\gamma,-\delta}(m) \sigma_{-\alpha,-\beta}(n)}{(mn)^{\frac12 }} \left(\frac{m}{n}\right)^{-it} \widetilde V_{ -\gamma, -\delta, -\alpha, -\beta,} \left(mn,t\right) + O_A((1+|t|)^{-A}),\nonumber
\end{align}
for any fixed $A>0$.
\end{lemma}
\begin{proof}
See Proposition 2.1 of~\cite{HY}.
\end{proof}

\noindent\textbf{Remarks.}\begin{itemize}
\item As mentioned in \cite{HY}, it is convenient to prescribe certain conditions on the function $G(s)$. To be precise, we assume $G(s)$ is divisible by an even polynomial $Q_{\alpha,\beta,\gamma,\delta}(s)$, which is symmetric in the parameters $\alpha,\beta,\gamma,\delta$, invariant under the transformations $\alpha\rightarrow-\alpha$, $\beta\rightarrow-\beta$, etc. and zero at $s=-\frac{(\alpha+\gamma)}{2}$ (as well as other points by symmetry), and that $G(s)/Q_{\alpha,\beta,\gamma,\delta}(s)$ is independent of $\alpha,\beta,\gamma,\delta$. An admissible choice is $Q_{\alpha,\beta,\gamma,\delta}(s)\exp(s^2)$ for such $Q_{\alpha,\beta,\gamma,\delta}(s)$, but there is no need to specify a particular function $G(s)$.

\item For $t$ large and $s$ in any fixed vertical strip Stirling's approximation gives
\begin{equation}\label{Xbound}
X_{\alpha,\beta,\gamma,\delta}(t)=\Big(\frac{t}{2\pi}\Big)^{-\alpha-\beta-\gamma-\delta}\big(1+O(t^{-1})\big)
\end{equation}
and
\begin{equation}\label{g}
g_{\alpha, \beta, \gamma, \delta}(s,t)=\Big(\frac{t}{2}\Big)^{2s}\Big(1+O\big(t^{-1}(1+|s|^2)\big)\Big).
\end{equation}
Moreover, for any fixed $A>0$ we have
\begin{equation}\label{V(x,t)}
t^j\frac{\partial^j}{\partial t^j}V_{\alpha, \beta, \gamma, \delta} \left(x,t \right)\ll_{A,j}(1+|x|/t^2)^{-A} .
\end{equation}
\end{itemize}

\subsection{Initial manipulations}

Applying the approximate functional equation~\eqref{afeq}, we see that
\est{
I_{\alpha,\beta,\gamma,\delta}(T)=J_{\alpha,\beta,\gamma,\delta}(T)+\widetilde J_{-\gamma,-\delta,-\alpha,-\beta}(T)+ O_A(T^{-A}),
}
for any fixed $A>0$, where
\est{
J_{\alpha,\beta,\gamma,\delta}(T)&=\sum_{a,b\leq T^\vartheta}\sum_{m_1,m_2,n_1,n_2}\frac {\alpha_a \overline {\alpha_{b}}}{(ab)^{\frac12}m_{1}^{\frac12+\alpha}m_{2}^{\frac12+\beta}n_{1}^{\frac12+\gamma}n_{2}^{\frac12+\delta}}\\
&\qquad\qquad\int_\R \Big(\frac{am_1m_2}{bn_1n_2}\Big)^{-it} V_{\alpha, \beta, \gamma, \delta} \left(m_1m_2n_1n_2,t \right)
 \Phi\Big(\frac tT\Big)\,dt
}
and $\widetilde J$ is the same sum, but with $\widetilde V$ in place of $V$. Write
\est{
J_{\alpha,\beta,\gamma,\delta}(T)=\mathcal M_{1;\alpha,\beta,\gamma,\delta}(T)+J^{*}_{\alpha,\beta,\gamma,\delta}(T)
}
and
\est{
\widetilde{J}_{-\gamma,-\delta,-\alpha,-\beta}(T)=\mathcal M_{2;-\gamma,-\delta,-\alpha,-\beta}(T)+\widetilde{J}^{*}_{-\gamma,-\delta,-\alpha,-\beta}(T),
}
where 
\est{
\mathcal M_{1;\alpha,\beta,\gamma,\delta}(T) &= \sum_{a,b\leq T^\vartheta}\sum_{\substack{m_1,m_2,n_1,n_2\\ am_1m_2=bn_1n_2}}\frac {\alpha_a \overline {\alpha_{b}}}{(ab)^{\frac12}m_{1}^{\frac12+\alpha}m_{2}^{\frac12+\beta}n_{1}^{\frac12+\gamma}n_{2}^{\frac12+\delta}}\\
&\qquad\qquad\int_\R  V_{\alpha, \beta, \gamma, \delta} \left(m_1 m_2 n_1 n_2 , t \right)
 \Phi\Big(\frac tT\Big)\,dt,
}
\est{
J^{*}_{\alpha,\beta,\gamma,\delta}(T)&=\sum_{a,b\leq T^\vartheta}\sum_{\substack{m_1,m_2,n_1,n_2\\ am_1m_2-bn_1n_2=h\neq0}}\frac {\alpha_a \overline {\alpha_{b}}}{(ab)^{\frac12}m_{1}^{\frac12+\alpha}m_{2}^{\frac12+\beta}n_{1}^{\frac12+\gamma}n_{2}^{\frac12+\delta}}\\
&\qquad\qquad\int_\R \Big(1+\frac{h}{bn_1n_2}\Big)^{-it} V_{\alpha, \beta, \gamma, \delta} \left(m_1m_2n_1n_2,t \right)
 \Phi\Big(\frac tT\Big)\,dt
}
and $\mathcal M_{2;-\gamma,-\delta,-\alpha,-\beta}$ and $\widetilde{J}^{*}_{-\gamma,-\delta,-\alpha,-\beta}$ being similar expressions.

\subsection{The diagonal terms}

As in Hughes and Young [\ref{HY}; Proposition 3.1] we have
 \begin{align*}
\mathcal M_{1;\alpha,\beta,\gamma,\delta}(T)= \sum_{g} \sum_{(a,b) = 1} \frac{\alpha_{ga} \overline{\beta_{gb}}}{g ab}\int_\R\Phi\Big(\frac tT\Big) Z_{\alpha, \beta, \gamma, \delta,a,b} dt+O_\eps(T^{\frac12+\frac\vartheta2+\eps}).
\end{align*}
Notice that when moving the line of integration to $\tRe(s)=-1/4+\eps$ in their equation (47), we cross only a simple pole at $s=0$. This is because of the cancellation of the zeros of the function $G(s)$ at $-\frac{(\alpha+\gamma)}{2}$, etc. with the poles of the zeta-functions in the formula. 

Similarly,
 \begin{align*}
\mathcal M_{2;-\gamma,-\delta,-\alpha,-\beta}(T)=&\sum_{g} \sum_{(a,b) = 1} \frac{\alpha_{ga} \overline{\beta_{gb}}}{g ab} \int_\R\Phi\Big(\frac tT\Big)\Big ( \frac{t}{2\pi} \Big )^{-\alpha - \beta - \gamma - \delta}  Z_{-\gamma, -\delta, -\alpha, -\beta,a,b}dt\\
&\qquad\qquad +O_\eps(T^{\frac12+\frac\vartheta2+\eps}).
\end{align*}

\subsection{The off-diagonal terms}

We first evaluate $J^*_{\alpha,\beta,\gamma,\delta}$. In view of~\eqref{V(x,t)}, the summands in $J^*_{\alpha,\beta,\gamma,\delta}(T)$ with $m_1m_2n_1n_2\gg T^{2+\eps}$ give a negligible contribution. Also, by integration by parts we have
\est{
\int_\R \Big(1+\frac{h}{bn_1n_2}\Big)^{-it} V_{\alpha, \beta, \gamma, \delta} \left(m_1m_2n_1n_2,t \right) \Phi\Big(\frac tT\Big)\,dt \ll_j \frac{T}{(h/\sqrt{abm_1m_2n_1n_2})^jT^j}
}
for any fixed $j\geq0$. So the contribution of the terms with $|h|>\sqrt{abm_1m_2n_1n_2}\,T^{-1+\eps}$ is $O_A(T^{-A})$ for any fixed $A>0$. Hence
\est{
J^{*}_{\alpha,\beta,\gamma,\delta}(T)&=\sum_{a,b\leq T^\vartheta}\sum_{\substack{m_1m_2n_1n_2\ll T^{2+\varepsilon}\\ am_1m_2-bn_1n_2=h\\0<|h|\leq\sqrt{abm_1m_2n_1n_2}\,T^{-1+\eps}}}\frac {\alpha_a \overline {\alpha_{b}}}{(ab)^{\frac12}m_{1}^{\frac12+\alpha}m_{2}^{\frac12+\beta}n_{1}^{\frac12+\gamma}n_{2}^{\frac12+\delta}}\\
&\qquad\qquad\int_\R \Big(1+\frac{h}{bn_1n_2}\Big)^{-it} V_{\alpha, \beta, \gamma, \delta} \left(m_1m_2n_1n_2,t \right)
 \Phi\Big(\frac tT\Big)\,dt+O_A(T^{-A}).
}

Note that a trivial bound gives
\begin{eqnarray}\label{trbound}
J^{*}_{\alpha,\beta,\gamma,\delta}(T)&\ll_\varepsilon& T^{1+\varepsilon}\sum_{a,b\leq T^\vartheta}\sum_{\substack{m_1m_2n_1n_2\ll T^{2+\varepsilon}\\ am_1m_2-bn_1n_2=h\\0<|h|\leq\sqrt{abm_1m_2n_1n_2}\,T^{-1+\eps}}}\frac{1}{(abm_1m_2n_1n_2)^\frac12}+O_A(T^{-A})\nonumber\\
&\ll_\varepsilon& T^{1+\vartheta+\varepsilon},
\end{eqnarray}
where the last estimate comes from letting $a,m_1,m_2$ and $h$ vary freely and bounding the number of values of $b,n_1,n_2$ by the number of divisors of $am_1m_2-h$. For $|h|\leq\sqrt{abm_1m_2n_1n_2}\,T^{-1+\eps}$, we have
\est{
\Big(1+\frac{h}{bn_1n_2}\Big)^{-it} =\textrm{e}\Big(-\frac{th}{2\pi bn_1n_2}\Big)+O_\varepsilon(T^{-1+\eps}).
}
Thus, using the trivial bound~\eqref{trbound} we get
\est{
J^{*}_{\alpha,\beta,\gamma,\delta}(T)&=\sum_{a,b\leq T^\vartheta}\sum_{\substack{m_1m_2n_1n_2\ll T^{2+\eps}\\ am_1m_2-bn_1n_2=h\neq0}}\frac {\alpha_a \overline {\alpha_{b}}}{(ab)^{\frac12}m_{1}^{\frac12+\alpha}m_{2}^{\frac12+\beta}n_{1}^{\frac12+\gamma}n_{2}^{\frac12+\delta}}\psi\Big(\frac{h^2T^{2-\eps}}{abm_1n_1m_2n_2}\Big)\\
&\qquad\qquad\int_\R \textrm{e}\Big(-\frac{th}{2\pi bn_1n_2}\Big) V_{\alpha, \beta, \gamma, \delta} \left(m_1m_2n_1n_2,t \right)
\Phi\Big(\frac tT\Big)\,dt+O_\varepsilon\pr{T^{\vartheta+\eps}},
}
where $\psi(x)$ is a function that is identically $1$ for $0\leq x\leq1$ and decays rapidly at infinity. 

Now, define
\es{\label{aeqac}
V^*(x)=\frac1{2\pi i}\int_{(1)}\frac{G(s)}{s}(2\pi)^{-2s}x^{-s}\,ds.
}
The estimate~\eqref{g} implies that $V_{\alpha,\beta,\gamma,\delta}(x,t)=V^*(x/t^2)+O_\varepsilon(t^{-1+2\eps}x^{-\eps})$. In particular, we can replace $V_{\alpha, \beta, \gamma, \delta}(x,t)$ with $V^*(x/t^2)$ in the above expression at the cost of an error of size $O_\varepsilon(T^{\vartheta+\eps})$. Grouping the terms $h$ and $-h$ allows us to replace
$\e{- t {h}/{2 \pi b n_1 n_2} }$ by $2 \cos ( t {h}/{b n_1 n_2}  )$ and the condition $h \neq 0$ is now replaced by $h > 0$. Thus
\begin{align*}
J^*_{\alpha,\beta,\gamma,\delta}(T) = & 2 \int_{\mathbb{R}}\sum_{a,b \leq T^{\theta}} \sum_{\substack{a m_1 m_2 - b n_1 n_2 = h > 0}}
\frac{\alpha_a \overline{\alpha_b} }{(ab)^{\frac12}m_{1}^{\frac12+\alpha}m_{2}^{\frac12+\beta}n_{1}^{\frac12+\gamma}n_{2}^{\frac12+\delta}}\psi\Big(\frac{h^2T^{2-\eps}}{abm_1n_1m_2n_2}\Big)
 \\
&\qquad\qquad  \cos \Big ( \frac{t h}{b n_1 n_2} \Big )  V^*\Big(\frac{m_1 m_2 n_1 n_2}{t^2}\Big) \Phi \Big ( \frac{t}{T} \Big ) dt + O_\varepsilon(T^{\vartheta + \varepsilon}).
\end{align*}

To the inner sum we apply our result on the quadratic divisor problem in the form of Theorem~\ref{thmqdp} (using partial summation before and after applying the theorem) with
\begin{align*}
f(x,y, z) & = 
\cos \Big ( \frac{t z}{y} \Big ) \psi\Big(\frac{z^2T^{2-\eps}}{xy}\Big),\quad  K(x)=V^*\Big(\frac{x}{t^2}\Big),\quad Z=T^{1-\eps}\quad\textrm{and}\quad X=t.
\end{align*}
We then get four main terms
$$
J^*_{\alpha,\beta,\gamma,\delta}(T) =\mathcal{M}^*_{\alpha,\beta,\gamma,\delta}(T) + \mathcal{M}^*_{\beta,\alpha,\gamma,\delta}(T)
+ \mathcal{M}^*_{\alpha,\beta,\delta,\gamma}(T) + \mathcal{M}^*_{\beta,\alpha,\delta,\gamma}(T)+\mathcal E  $$
with the error being bounded by
\est{
\mathcal E  %&\ll T^{\eps}\frac{T}{(ABX^2)^\frac12} \frac{(AB)^\frac12 X}{Z^\frac12} \cdot \Big ( AB + (A^{1/2} + B^{1/2}) \frac{(AB)^\frac14X^\frac12}{Z^\frac14}\Big )\\
%&\ll T^{\frac12+\eps}(T^{2\vartheta}+T^{\vartheta}T^{\frac 14})\\
&\ll_\varepsilon T^{\frac12+\eps}(T^{2\vartheta}+T^{\frac 14+\vartheta}).
}

Let us focus on the first main term.
We have 
\est{%\label{arac}
 \mathcal{M}^*_{\alpha,\beta,\gamma,\delta}(T) &= 2\frac{\zeta(1 + \alpha - \beta)\zeta(1 +\gamma - \delta)}{\zeta(2 + \alpha - \beta + \gamma - \delta)}\sum_{g}
\sum_{(a,b) = 1} \frac{\alpha_{ga} {\beta_{gb}}g}{(g a )^{1  - \beta} (g b)^{1 -\delta }} 
 \\ &\quad\int_\R \int_{0}^{\infty} V^*\Big(\frac{x^2}{t^2g^2ab}\Big)\widetilde{f}_{\alpha,\beta, \gamma,\delta}(x, x; a,b,g) x^{-1-\beta -\delta} \Phi \Big ( \frac{t}{T} \Big )dx\, dt,
}
where $\widetilde{f}_{ \alpha, \beta,\gamma, \delta}(x, x; a,b,g)$ is equal to 
\begin{align*}
\frac{1}{2\pi i} \int_{(1+\eps)}\widehat{f}_3(x, x, s)\zeta(s) \zeta(1 + \alpha - \beta+ \gamma - \delta  + s) g^{- s} \eta_{\alpha,\beta,\gamma,\delta,a,b}(0,0,s) 
ds,
\end{align*}
$\widehat{f}_3$ is the Mellin transform of $f(x,y,z)$ with respect to $z$ and $ \eta_{\alpha,\beta,\gamma,\delta,a,b}(u,v,s)$ is a finite Euler product defined as in~\eqref{etanew}. After a change of variable we have
\est{
\widehat{f}_3(x,x,s)=x^{s}\int _{0}^{\infty} \cos  (t u ) \psi(u^2T^{2-\eps})u^{s-1}\,du.
} 
The integral over $u$ can be expressed as a convolution of Mellin transform\footnote{Since $\int_{0}^\infty\cos(tx)x^{w-1}\,dx=t^{-w}\Gamma(w)\cos(\frac{\pi w}{2})$ for $0<\Re(w)<1$ and $t>0$.}, so
\est{
\widehat{f}_3(x,x,s)=x^s\frac1{2\pi i}\int_{(0)}\widehat\psi(z)T^{-(2-\eps)z}\Gamma(s-2z)\cos\Big(\frac{\pi }{2}(s-2z)\Big)t^{2z-s}\,dz.
}
We move the line of integration to $\tRe(z)=-A$ for some large $A>0$, collecting a residue at $z=0$ only (since $\widehat\psi(z)$ has a simple pole of residue $1$ at $z=0$). Taking $A$ large enough with respect to $\eps$ we obtain
\est{
\widehat{f}_3(x,x,s)=x^{s}\Gamma(s)\cos\Big(\frac{\pi s}{2}\Big)t^{-s}+O_{A,\eps}(x^{1+\eps}T^{-A}),
}
since $t\asymp T$. We can ignore the $O$-term as this contributes an error of size $O_A(T^{-A})$.%Because error term is $x^{\Re(s)}T^{-(1-\eps/2)\Re(s)} (T/t)^{\Re(z)} t^{-\eps/2 \Re(z)}$ and the last factor is $O(T^{-A}) if $\Re(z)$ is large enough.

Now we evaluate the integral over $x$ obtaining
\est{
  \int_{0}^{\infty} V^*\Big(\frac{x^2}{t^2g^2ab}\Big) x^{- 1 - \beta - \delta+s} dx = (t g \sqrt{a b})^{ - \beta - \delta+s}
 (2\pi)^{ \beta+ \delta-s} \frac{G \pr{\frac{ - \beta - \delta+s}{2}}}{ - \beta - \delta+s} 
}
by the Mellin expression~\eqref{aeqac} for $V^*(x)$. Thus, we obtain
\est{
 \mathcal{M}^*_{\alpha,\beta,\gamma,\delta}(T)=\mathcal{M}^{**}_{\alpha,\beta,\gamma,\delta}(T)+O_A(T^{-A}),
}
where
\est{%\label{asdfsa}
 \mathcal{M}^{**}_{\alpha,\beta,\gamma,\delta}(T) &= 2\sum_{g}
\sum_{(a,b) = 1} \frac{\alpha_{ga} {\beta_{gb}}}{g  a ^{1  - \beta} b^{1 -\delta }}
\frac{\zeta(1 + \alpha - \beta)\zeta(1 +\gamma - \delta)}{\zeta(2 + \alpha - \beta + \gamma - \delta)}\\
&\frac{1}{2\pi i}\int_\R \int_{(1+\eps)}(t/2\pi)^{-\beta-\delta}\frac{G \pr{\frac{ - \beta - \delta+s}{2}}}{ - \beta - \delta+s}(ab)^{(-\beta-\delta+s)/2}(2\pi)^{-s}\Gamma(s)\cos\pr{\frac{\pi s}{2}}\\
&\qquad\qquad  \zeta(s) \zeta(1 + \alpha - \beta+ \gamma - \delta  + s) \eta_{\alpha,\beta,\gamma,\delta,a,b}(0,0,s)\Phi \Big ( \frac{t}{T} \Big ) dsdt.
}
Applying the functional equation $\zeta(1-s)=2(2\pi)^{-s}\Gamma(s)\cos\pr{\frac{\pi s}{2}}\zeta(s)$ and making the change of variable $s\rightarrow \beta+\delta+2s$ we arrive to
 \est{%\label{asdfsa}
 \mathcal{M}^{**}_{\alpha,\beta,\gamma,\delta}(T) &= \frac{\zeta(1 + \alpha - \beta)\zeta(1 +\gamma - \delta)}{\zeta(2 + \alpha - \beta + \gamma - \delta)}\\
&\qquad\qquad\qquad\frac{1}{2\pi i}\int_\R\Phi \Big ( \frac{t}{T} \Big )(t/2\pi)^{-\beta-\delta} \int_{(1+\eps)}\mathfrak{M}_{\alpha,\beta,\gamma,\delta}(s)\frac{G(s)}{ s} dsdt,
}
where
 \begin{eqnarray*}%\label{asdfsa}
 \mathfrak{M}_{\alpha,\beta,\gamma,\delta}(s) &=& \zeta(1 + \alpha+ \gamma   + 2s)\zeta(1-\beta-\delta-2s)\nonumber
\\
&&\qquad   \sum_{g}
\sum_{(a,b) = 1} \frac{\alpha_{ga} {\beta_{gb}}}{g  a ^{1  - \beta-s} b^{1 -\delta-s }} \eta_{\alpha,\beta,\gamma,\delta,a,b}(0,0,\beta+\delta+2s).
\end{eqnarray*}
In summary we have
$$
J^*_{\alpha, \beta, \gamma, \delta}(T) = \mathcal{M}^{**}_{\alpha, \beta, \gamma, \delta} (T)+ \mathcal{M}^{**}_{\beta, \alpha, \gamma, \delta}(T)
+ \mathcal{M}^{**}_{\alpha, \beta, \delta, \gamma} (T)+ \mathcal{M}^{**}_{\beta, \alpha, \delta, \gamma}(T) + \mathcal{E},
$$
where $\mathcal{E} \ll_\eps T^{\frac12 + \varepsilon} (T^{2 \vartheta} + T^{\vartheta + \frac14})$.

On the other hand, proceeding identically to the above we also
find that
\begin{align*}
  \widetilde{J}_{-\gamma, -\delta, -\alpha, -\beta}^{*}(T) =& \widetilde{\mathcal{M}}^{**}_{-\gamma, -\delta, -\alpha, -\beta}(T)
+ \widetilde{\mathcal{M}}^{**}_{-\delta, -\gamma, -\alpha, -\beta}(T)\\
&\qquad\qquad + \widetilde{\mathcal{M}}^{**}_{-\gamma, -\delta, -\beta, -\alpha}(T)
+ \widetilde{\mathcal{M}}^{**}_{-\delta, -\gamma, -\beta, -\alpha}(T) + \widetilde{\mathcal{E}},
\end{align*}
where $\widetilde{\mathcal{E}} \ll_\eps  T^{\frac12 + \varepsilon} (T^{2 \vartheta} + T^{\vartheta + \frac14})$ and where, for example,
 \est{%\label{asdfsa}
&\widetilde{\mathcal{M}}^{**}_{-\delta, -\gamma, -\beta, -\alpha}(T) = \frac{\zeta(1 + \alpha - \beta)\zeta(1 +\gamma - \delta)}{\zeta(2 + \alpha - \beta + \gamma - \delta)}\\
&\qquad\qquad\frac{1}{2\pi i}\int_\R\Phi \Big ( \frac{t}{T} \Big )X_{\beta,\alpha,\delta,\gamma}(t)(t/2\pi)^{\alpha+\gamma} \int_{(1+\eps)}\frac{G(s)}{ s}\mathfrak{M}_{-\delta, -\gamma, -\beta, -\alpha}(s) dsdt.
}
In view of~\eqref{Xbound} we get
 \est{%\label{asdfsa}
& \mathcal{M}^{**}_{\alpha, \beta, \gamma, \delta}(T)+\widetilde{\mathcal{M}}^{**}_{-\delta, -\gamma, -\beta, -\alpha}(T) = \frac{\zeta(1 + \alpha - \beta)\zeta(1 +\gamma - \delta)}{\zeta(2 + \alpha - \beta + \gamma - \delta)}\frac{1}{2\pi i}\int_\R\Phi \Big ( \frac{t}{T} \Big )(t/2\pi)^{-\beta-\delta} \\
&\qquad\qquad\int_{(1+\eps)}\frac{G(s)}{ s}\Big(\mathfrak{M}_{\alpha,\beta,\gamma,\delta}(s)+\mathfrak{M}_{-\delta, -\gamma, -\beta, -\alpha}(s) \Big)dsdt+O_\eps(T^\eps).
}
It is a standard exercise to check that $\mathfrak{M}_{\alpha,\beta,\gamma,\delta}(s)=\mathfrak{M}_{-\delta, -\gamma, -\beta, -\alpha}(-s)$. Hence by the residue theorem, noticing that the only pole in the strip $-(1+\eps)\leq\tRe(s)\leq 1+\eps$ is at $s=0$ as we assume that the function $G(s)$ vanishes at $-\frac{(\alpha+\gamma)}{2}$ and $-\frac{(\beta+\delta)}{2}$,
 \est{%\label{asdfsa}
 \mathcal{M}^{**}_{\alpha, \beta, \gamma, \delta}(T)+\widetilde{\mathcal{M}}^{**}_{-\delta, -\gamma, -\beta, -\alpha}(T)& = \frac{\zeta(1 + \alpha - \beta)\zeta(1 +\gamma - \delta)}{\zeta(2 + \alpha - \beta + \gamma - \delta)} \\
&\qquad\quad\frac{1}{2\pi i}\int_\R\Phi \Big ( \frac{t}{T} \Big )(t/2\pi)^{-\beta-\delta}\mathfrak{M}_{\alpha,\beta,\gamma,\delta}(0)dt+O_\eps(T^\eps).
}

The other terms combine in the same way. Hence we are left to show that
\[
\frac{\zeta(1 + \alpha - \beta)\zeta(1 +\gamma - \delta)}{\zeta(2 + \alpha - \beta + \gamma - \delta)}\mathfrak{M}_{\alpha,\beta,\gamma,\delta}(0)=\sum_{g} \sum_{(a,b) = 1} \frac{\alpha_{ga} \overline{\beta_{gb}}}{g ab}Z_{\alpha,-\delta,\gamma,-\beta,a,b},
\]
which reduces to
\[
a^\beta b^\delta \eta_{\alpha,\beta,\gamma,\delta,a,b}(0,0,\beta+\delta)=B_{\alpha,-\delta,\gamma,-\beta,a}B_{\gamma,-\beta,\alpha,-\delta,b}.
\]
By symmetry and multiplicativity, this is equivalent to
\begin{equation}\label{verify}
p^{\nu\beta} \eta_{\alpha,\beta,\gamma,\delta,p^\nu}(0,0,\beta+\delta)=B_{\alpha,-\delta,\gamma,-\beta,p^\nu}.
\end{equation}

From Lemma 6.9 of \cite{HY} we have
\[
B_{\alpha,-\delta,\gamma,-\beta,p^\nu}=\Big(1-\frac{1}{p^{2+\alpha-\beta+\gamma-\delta}}\Big)^{-1}\frac{p^{-\beta}}{p^{-(\beta+\gamma)}-1}\Big(B^{(0)}-p^{-1}B^{(1)}+p^{-2}B^{(2)}\Big),
\]
where
\begin{align*}
&B^{(0)}=p^{-(\nu+1)\gamma}-p^{(\nu+1)\beta},\\
&B^{(1)}=(p^{-\alpha}+p^\delta)p^{\beta-\gamma}(p^{-\nu\gamma}-p^{\nu\beta}),\\
&B^{(2)}=p^{-\alpha+\beta-\gamma+\delta}(p^{\beta-\nu\gamma}-p^{\nu\beta-\gamma}).
\end{align*}
On the other hand, using the definition of $\eta_{\alpha,\beta,\gamma,\delta,a}(u,v,s)$ in~\eqref{etanew}, the left hand side in~\eqref{verify} is equal to
\begin{align*}
&p^{\nu\beta}\bigg(p^{-\nu(\beta+\gamma)}+\sum_{0\leq j<\nu}p^{-j(\beta+\gamma)}c_p(\alpha+\gamma,\gamma-\delta,\alpha-\beta+\gamma-\delta)\bigg)\\
&\qquad=p^{-\nu\gamma}+\frac{p^{-\nu\gamma}-p^{\nu\beta}}{p^{-(\beta+\gamma)}-1}\Big(1-\frac{1}{p^{1+\alpha+\gamma}}\Big)\Big(1-\frac{1}{p^{1+\gamma-\delta}}\Big)\Big(1-\frac{1}{p^{2+\alpha-\beta+\gamma-\delta}}\Big)^{-1}.
\end{align*}
So~\eqref{verify} is equivalent to
\begin{align*}
&p^{-\nu\gamma}(p^{-(\beta+\gamma)}-1)\big(1-p^{-2}p^{-\alpha+\beta-\gamma+\delta}\big)+(p^{-\nu\gamma}-p^{\nu\beta})\big(1-p^{-1}p^{-(\alpha+\gamma)}\big)\Big(1-p^{-1}p^{-\gamma+\delta}\big)\\
&\qquad\qquad=p^{-\beta}\Big(B^{(0)}-p^{-1}B^{(1)}+p^{-2}B^{(2)}\Big).
\end{align*}
It is an easy exercise to check that the above holds by comparing the coefficients of $p^0,p^{-1}$ and $p^{-2}$, and hence Theorem~\ref{mt} follows.

\subsection{Proof of Theorem~\ref{variantmt}}

In the remaining of the section, we shall show that
\begin{equation}\label{ZtildeZ}
Z_{\alpha,\beta,\gamma,\delta,a,b}=ab\widetilde{Z}_{\alpha,\beta,\gamma,\delta,a,b}(t)+O_\eps\big(T^{-(1-\vartheta)/2+\eps}\big)
\end{equation}
for $t\asymp T$, $a,b\leq T^\vartheta$ and $(a,b)=1$, and hence Theorem~\ref{mt} will imply Theorem~\ref{variantmt}.

From~\eqref{aeqac} we have
\[
\widetilde{Z}_{\alpha,\beta,\gamma,\delta,a,b}(t)=\frac1{2\pi i}\int_{(1)}\frac{G(s)}{s}\Big(\frac{t}{2\pi}\Big)^{2s}\sum_{am=bn}\frac{\sigma_{\alpha,\beta}(m)\sigma_{\gamma,\delta}(n)}{(ab)^{\frac12}(mn)^{\frac12+s}}\,ds.
\]
Since $(a,b)=1$ we get
\begin{equation}\label{tildeZ}
\widetilde{Z}_{\alpha,\beta,\gamma,\delta,a,b}(t)=\frac1{2\pi i}\int_{(1)}\frac{G(s)}{s}\Big(\frac{t}{2\pi}\Big)^{2s}(ab)^{-(1+s)}\sum_{n=1}^{\infty}\frac{\sigma_{\alpha,\beta}(bn)\sigma_{\gamma,\delta}(an)}{n^{1+2s}}\,ds.
\end{equation}
Let
\begin{align*}
A_{\alpha,\beta,\gamma,\delta}(s)&=\sum_{n=1}^{\infty}\frac{\sigma_{\alpha,\beta}(n)\sigma_{\gamma,\delta}(n)}{n^{1+2s}}\\
&=\frac{\zeta(1 + \alpha + \gamma+2s) \zeta(1 +\alpha + \delta+2s) \zeta(1 + \beta + \gamma+2s)\zeta(1+\beta+\delta+2s)}{\zeta(2 + \alpha + \beta + \gamma + \delta+4s)}
\end{align*}
and
\est{
B_{\alpha,\beta,\gamma,\delta,a}(s)=\prod_{p^\nu||a}\left(\frac{\sum_{j=0}^{\infty}\sigma_{\alpha,\beta}(p^j)\sigma_{\gamma,\delta}(p^{j+\nu})p^{-j(1+2s)}}{\sum_{j=0}^{\infty}\sigma_{\alpha,\beta}(p^j)\sigma_{\gamma,\delta}(p^{j})p^{-j(1+2s)}}\right),
}
so that $A_{\alpha,\beta,\gamma,\delta}=A_{\alpha,\beta,\gamma,\delta}(0)$, $B_{\alpha,\beta,\gamma,\delta,a}=B_{\alpha,\beta,\gamma,\delta,a}(0)$ and
\[
\sum_{n=1}^{\infty}\frac{\sigma_{\alpha,\beta}(bn)\sigma_{\gamma,\delta}(an)}{n^{1+2s}}=A_{\alpha,\beta,\gamma,\delta}(s)B_{\alpha,\beta,\gamma,\delta,a}(s)B_{\gamma,\delta,\alpha,\beta,b}(s)
\]
Moving the line of integration in~\eqref{tildeZ} to $\tRe(s)=-1/4+\eps$, we cross only a simple pole at $s=0$. The zeros of $G(s)$ at $-\frac{(\alpha+\gamma)}{2}$, etc. cancel out various poles of the zeta-functions. Bounding the new integral by absolute values we obtain
\[
\widetilde{Z}_{\alpha,\beta,\gamma,\delta,a,b}(t)=(ab)^{-1}A_{\alpha,\beta,\gamma,\delta}B_{\alpha,\beta,\gamma,\delta,a}B_{\gamma,\delta,\alpha,\beta,b}+O_\eps\Big(T^{-\frac12+\eps}(ab)^{-\frac34}\Big)
\]
and so~\eqref{ZtildeZ} follows.

\section{An unbalanced quadratic divisor problem}
As preparation for the proof of our quadratic divisor problem (Theorem~\ref{qdp}) we consider
the first an ``unbalanced'' divisor problem where the variables $m_1, m_2, n_1, n_2$ appearing in
$a m_1 m_2 - b n_1 n_2 = h$ are (essentially) subject to the condition that $m_1 < m_2$ and
$n_1 < n_2$. This assumption simplifies the decision on which variable to apply Poisson summation formula. 
In the proof of this result we appeal to our main technical ingredients: Watt's theorem and the Weil bound. 
\begin{proposition}\label{dsff}
Let $A,B,M_1,M_2,N_1,N_2,H\geq1$ and let $M=M_1M_2$, $N=N_1N_2$. Let $W_i$, for $i=0,1,\dots,4$, be smooth functions supported in $[1,2]$ such that $W_i^{(j)}\ll_j (ABMN)^{\eps}$ for any fixed $j\geq0$. Let $\alpha_a,\beta_b$ be sequences of complex numbers supported on $[A,2A]$ and $[B,2B]$, respectively, and such that $\alpha_a\ll A^\eps, \beta_b\ll B^\eps$. Let
\est{
\mathcal S_\pm &=\sum_{am_1m_2-bn_1n_2=\pm h\ne0}\alpha_a\beta_b W_0\Big(\frac hH\Big)W_1\Big(\frac{m_1}{M_1}\Big)W_2\Big(\frac{m_2}{M_2}\Big)W_3\Big(\frac{n_1}{N_1}\Big)W_4\Big(\frac{n_2}{N_2}\Big),
}
where the sum runs over positive integers $a,b,m_1,m_2,n_1,n_2$ and $h$. Assume that we have $M_1\leq M_2 (ABM N)^{\varepsilon}$, $N_1\leq N_2 (ABM N)^{\varepsilon}$ and $H\ll (AB)^{\frac12+\eps}$. Then
\est{
\mathcal S_\pm&= \mathcal M+\mathcal E ,
}
where
\est{
\mathcal M&=\sum_{\substack{a,b,m_1,n_1,h,d\\(am_1,bn_1)=d}}\alpha_a\beta_bW_0\Big(\frac{dh}{H}\Big)W_1\Big(\frac{m_1}{M_1}\Big)W_3\Big(\frac{n_1}{N_1}\Big)\int_{0}^\infty  W_2\Big(\frac{bn_1x}{dM_2}\Big)W_4\Big(\frac{am_1x}{dN_2}\Big)\,dx}
and
\begin{eqnarray*}
\mathcal E&\ll_\varepsilon  \pr{ ABMNH^2}^{\frac14+\eps}\pr{AB+H^\frac14(A+B)^\frac12(ABMN)^\frac18}.
\end{eqnarray*}
Moreover, without any assumption on $H$ the same result holds with the bound for $\mathcal E$ being replaced by
\es{\label{bwa}
\mathcal E&\ll_\varepsilon (ABMNH^2)^{\frac38+\eps} (ABH)^\frac14 (A+B)^\frac54+(ABMN)^{\eps}H^2.
}
\end{proposition}
\begin{proof}
First, we observe that we can assume there is $\delta>0$ such that $MN\gg (AB)^{\delta}$ and, for~\eqref{bwa}, $H\ll (ABMN)^{\frac12-\delta}$ since otherwise the bound is trivial, and that $AM\asymp BN$ (otherwise the sum is empty when $AMBNH$ is large enough). %Because of our assumptions on $H$.
Moreover, by symmetry we can assume $BN_1\leq AM_1$. To summarize, we have
\es{\label{hypp}
AM\asymp BN,\qquad   MN\gg (AB)^{\delta}\qquad\textrm{and}\qquad BN_1\leq AM_1.
}

Now, let $d=(am_1,bn_1)$ (note that this implies $d|h$). We can eliminate the variable $n_2$ by writing $am_1m_2-bn_1n_2=\pm h$ as $m_2\equiv \pm(h/d)\overline{am_1/d}\mod{bn_1/d}$:
\begin{eqnarray*}%\label{kn1}
&&\sum_{\substack{m_2,n_2\\am_1m_2-bn_1n_2=\pm h}}W_2\Big(\frac{m_2}{M_2}\Big)W_4\Big(\frac{n_2}{N_2}\Big)\nonumber\\
&&\qquad=\sum_{m_2\equiv \pm(h/d)\overline{am_1/d}\mod{ bn_1/d}}W_2\Big(\frac{m_2}{M_2}\Big)W_4\Big(\frac{am_1m_2\mp h}{bn_1N_2}\Big)\\
&&\qquad=\sum_{m_2\equiv \pm(h/d)\overline{am_1/d}\mod{ bn_1/d}}W_2\Big(\frac{m_2}{M_2}\Big)W_4\Big(\frac{am_1m_2}{bn_1N_2}\Big)\Big(1+O_\varepsilon \big(H(AM)^{-1+\varepsilon}\big)\Big).\nonumber
\end{eqnarray*}
The contribution of the error term to $S_\pm$ is bounded by 
\est{
H(AM)^{-1+\varepsilon}\sum_{h\asymp H}\sum_{\substack{am_1m_2-bn_1n_2=\pm h,\\am_1m_2\asymp bn_1n_2 \asymp AM}}(AM)^\eps \ll_\varepsilon H^2 (AM)^\eps,
}
and, thus, after applying Poisson's summation and changing $h$ into $dh$, we get
\es{\label{asd}
\mathcal S_\pm &=\sum_{d\leq 2H}\sum_{\substack{a,b,m_1,n_1,h\\(am_1,bn_1)=d}}\sum_{l\in\Z}\alpha_a\beta_bW_0\Big(\frac {dh}H\Big)W_1\Big(\frac{m_1}{M_1}\Big)W_3\Big(\frac{n_1}{N_1}\Big)\\
&\qquad\qquad \textrm{e}\bigg(\mp lh\frac{\overline{am_1/d}}{bn_1/d}\bigg)F(a,b,m_1,n_1,d,l)+O_\varepsilon\pr{H^2(AM)^\eps},
}
where
\est{
F(a,b,m_1,n_1,d,l)&=\frac{d^2}{abm_1n_1}\int_{0}^\infty  W_2\Big(\frac{xd}{am_1M_2}\Big)W_4\Big(\frac{xd}{bn_1N_2}\Big)\e{\frac{d^2lx}{abm_1n_1}}\,dx\\
&=\int_{0}^\infty W_2\Big(\frac{bn_1x}{dM_2}\Big)W_4\Big(\frac{am_1x}{dN_2}\Big)\e{lx}\,dx.\\
}

The term $l=0$ corresponds to the main term (notice that the sum over $d$ can be
extended to an infinite sum since $W_0(\cdot)$ is compactly supported in $[1,2]$). 
For the terms with $l\ne0$, integration by parts implies
\begin{eqnarray*}
F(a,b,m_1,n_1,d,l)&\ll_\varepsilon&  (AM)^{\eps}\frac1{l^j} \Big(\frac{bn_1}{dM_2}+\frac{am_1}{dN_2}\Big)^j\frac{dM_2}{bn_1}\\
&\ll_\varepsilon&(AM)^{\eps} \Big(\frac{AM}{dlM_2N_2}\Big)^j\frac{dM_2}{BN_1}
\end{eqnarray*}
for any fixed $j\geq0$. Hence we can restrict the sum in~\eqref{asd} to $0<|l|\leq L$, where 
\est{%\label{bfl}
L=\frac{AM}{dM_2N_2}(AM)^\eps.
}
Thus, we have
\est{
\mathcal S_\pm=\mathcal M+\mathcal R_\pm+O_\varepsilon\pr{H^2(AM)^\eps},
}
where
\begin{eqnarray*}
\mathcal R_\pm&=&\sum_{d\leq 2H} \sum_{\substack{a,b,m_1,n_1,h\\(am_1,bn_1)=d}}\sum_{0<|l|\leq L}\alpha_a\beta_bW_0\Big(\frac {dh}H\Big) W_1\Big(\frac{m_1}{M_1}\Big)W_3\Big(\frac{n_1}{N_1}\Big) \\
&&\qquad\qquad \textrm{e}\bigg(\mp lh\frac{\overline{am_1/d}}{bn_1/d}\bigg)F(a,b,m_1,n_1,l,d).\\
\end{eqnarray*}

From the definition of $F$, we have
\est{
\mathcal R_\pm\ll \sum_{d\leq 2H} \int_{x\asymp \frac{dM_2}{BN_1}\asymp \frac{dN_2}{AM_1}}|Z_{\pm,d}(x)|\,dx,
}
where
\est{
Z_{\pm,d}(x)&=\sum_{\substack{a,b,m_1,n_1,h\\(am_1,bn_1)=d}}\sum_{0<|l|\leq L}\alpha_a\beta_bW_0\Big(\frac {dh}H\Big)W_1\Big(\frac{m_1}{M_1}\Big)W_3\Big(\frac{n_1}{N_1}\Big)\\
&\qquad\qquad W_2\Big(\frac{bn_1x}{dM_2}\Big)W_4\Big(\frac{am_1x}{dN_2}\Big)  \textrm{e}\bigg(\mp lh\frac{\overline{am_1/d}}{bn_1/d}\bigg)\e{lx}.
}
We can bound $Z_{\pm,d}$ using the following lemma which we will prove in the next subsection.

\begin{lemma}\label{bfz}
Under the conditions of Proposition~\ref{dsff} (without the condition $H\ll (AB)^{\frac12+\eps}$), the assumptions~\eqref{hypp} and $x\asymp \frac {dN_2}{AM_1}$, we have
\es{\label{fbfz}
Z_{\pm,d}(x)&\ll_\varepsilon % (AM)^\eps\frac{A^\frac32M_1^\frac12 HBN_1}{d^\frac72(N_2M_2)^\frac12}AM_1=
(AM)^\eps\frac{A^{\frac32}B^\frac12 H}{d^{\frac72}}\Big(\frac{M_1 N_1}{M_2N_2}\Big)^\frac12(BN_1)^\frac12 \Big(BN_1+M_1 \min \big\{A,H\big\}\Big).\\
}
Moreover, if $H\ll (AB)^{\frac12+\eps}$ and $d\ll (AB)^{\frac{1}{2}}(AM)^{-100\eps}$, then
\es{\label{sbfz}
Z_{\pm,d}(x)&\ll _\varepsilon%\frac{HAB}{d^2} \bigg(\frac{AM}{H}\bigg)^{1/2} \pr{\frac{ABM_1N_1}{N_2M_2}}^{\frac12 } \pr{1+\frac{HN_1^2}{B^2A^3}}^\frac14 (AM)^{\eps}\\
%&= 
(AM)^{\eps}\frac{A^2BH^\frac12}{d^2}  \Big(\frac{M_1N_1}{M_2N_2}\Big)^{\frac12 } (BM)^\frac12 \Big(1+\frac{N_1^2H}{A^3B^2}\Big)^\frac14.\\
}
\end{lemma}
We first assume that $H\ll (AB)^{\frac12+\eps}$. We apply~\eqref{fbfz} to the terms with $d>\min \Big\{(AB)^{\frac{1}{2}}(AM)^{-100\eps},\big(\frac{AM_1N_1H}{BM_2}\big)^\frac13\Big\}$. We integrate over $x \asymp d N_2 / (A M_1)$ and then use the inequality $\sum_{d > \min(z,w)} d^{-\frac52} \ll z^{-\frac32} + w^{-\frac32}$ getting that the contribution of these terms to $\mathcal R_\pm$ is
\est{
&\ll_\varepsilon (AM)^\eps\frac{(ABN)^\frac12 H}{M^\frac12}(BN_1)^\frac12\Big(BN_1+M_1\min \big\{A,H\big\}\Big)\\
&\qquad\qquad\qquad\qquad\qquad\qquad\qquad\qquad\bigg(\Big(\frac{BM_2}{AM_1N_1H}\Big)^\frac12+\frac1{\pr{AB}^\frac34}\bigg)%  Brackets: (I+II)(a+b)
}
We think of the above expression as being of the form $(I + II)(a + b)$, expanding it as $I\cdot a + I\cdot b + II\cdot a + II \cdot b$ we use the inequality $\min(A, H) \leq A$ and $\min(A,H) \leq H$ in the terms $II \cdot a$ and $II\cdot b$ respectively, getting, 
\est{
& \ll_\varepsilon  (AM)^\eps\Big(
 \frac{B^\frac52N^\frac12N_1H^\frac12}{M_1} %I*a  
+\frac{ B^{\frac54}N^\frac12N_1^\frac32H}{A^\frac14M^\frac12} % I*b
+AB^\frac32N^\frac12 H^\frac12 %II*a, using min(A,H)\leq A
+\frac{B^\frac14 M_1^\frac12(NN_1)^\frac12  H^2}{A^\frac14M_2^\frac12} %II*b, using min(A,H)\leq H
\Big)
}
Subsequently in the first term we use $B N_1 \leq A M_1$, in the second term we use $A M \asymp (B N)$ and in the fourth term $H \ll (A B)^{\frac12 + \varepsilon}$ together with $M_1 \ll M_2 (A M)^{\varepsilon}$, 
\est{
&\ll_\varepsilon (AM)^\eps \Big(
AB^\frac{3}{2}N^\frac12H^\frac12 %First of previous line less less than this because of $BN_1\leq AM_1$
+A^\frac14 B^\frac34 N_1^\frac32H %Second of previous line less less than this because $AM\asymp BN$
+A^\frac14B^\frac34 (NN_1)^\frac12H%Fourth od previous line less less than this because $H\ll (AB)^{\frac12+\eps}$ 
\Big)}
Finally using the inequalities $ AM\asymp BN$, 
$N_1\ll N_2(AM)^\eps$ and $H\ll (AB)^{\frac12+\eps}$ we conclude with the bound
\est{
&\ll_\varepsilon  \pr{ BNH}^{\frac12+\eps}\pr{AB+A^\frac14B^\frac14N_1^\frac12H^\frac12}
\ll_\varepsilon  \pr{ BNH}^{\frac12+\eps}\pr{AB+A^\frac14(BN)^\frac14H^\frac12}\\
&\ll_\varepsilon  \pr{ ABMNH^2}^{\frac14+\eps}\pr{AB+H^\frac14A^\frac38B^\frac18(ABMN)^\frac18}.
}
For the other values of $d$ we apply~\eqref{sbfz}. The integration over $x$ contributes $d M_2 / (B N_1)$, while the sum over $d$ is bounded using $\sum_{d \leq (A B)^{\frac12}} d^{-1} \ll_\eps (A B)^{\varepsilon}$. Thus the contribution of these terms to $\mathcal R_\pm$ is
\est{
&\ll_\varepsilon(AM)^{\eps}\frac{AM_2}{BN_1}ABH^\frac12  \Big(\frac{M_1N_1}{M_2N_2}\Big)^{\frac12 }(BM)^\frac12 \Big(1+\frac{N_1^2H}{A^3B^2}\Big)^\frac14 
}
Repeatedly using that $A M \asymp B N$ we see that the above is 
\est{
&\ll_\varepsilon ABH^\frac12  \pr{AM}^{\frac12 +\eps}%less less than first addend because $AM\asymp BN$
+ A^\frac14B^\frac12 H^\frac34 N_1^\frac12  \pr{AM}^{\frac12 +\eps } \\%less less than first addend because $AM\asymp BN$
&\ll_\varepsilon  \pr{AMH}^{\frac12 +\eps}\Big(AB 
+ (ABH)^\frac14(ABMN)^\frac18 %less less than first addend because $AM\asymp BN$
\Big),\\
}
and so Proposition~\ref{dsff} follows in the case $H\ll (AB)^{\frac12+\eps}$. 

Without the assumption $H\ll (AB)^{\frac12+\eps}$ we apply~\eqref{fbfz} for all $d$, integrating over $x \asymp d N_2 / (A M_1)$ and obtain
\est{
\mathcal R_\pm %&\ll \frac{N_2}{AM_1} 
%(AM)^\eps{A^{\frac32}B^\frac12 H}\pr{\frac{M_1 N_1}{N_2M_2}}^\frac12((BN_1)^\frac32+AM_1 \min (H^\frac12,(BN_1)^\frac12))\\
&\ll_\varepsilon  
(AM)^\eps A^{\frac12}B^\frac12 \Big(\frac{ N}{M}\Big)^\frac12H(BN_1)^\frac12\big(BN_1+AM_1\big)\\
&\ll_\varepsilon  
(AM)^\eps AH (ABM_1N_1)^\frac12(BN_1+AM_1)^\frac12, %since $BN_1\leq AM_1$
\\
}
since $ AM\asymp BN$ and $BN_1\leq AM_1$. Finally, since $M_1\ll M_2(AM)^\eps$ and $ AM\asymp BN$,  we have $AM_1\ll A^{\frac12}(AM)^{\frac12}(AM)^\eps\asymp A^{\frac12}(ABMN)^{\frac14}(AM)^\eps$ and similarly for $BN_1$, thus
\est{
\mathcal R_\pm &\ll_\varepsilon  
(AM)^\eps AH(AB)^\frac14(A+B)^\frac14 (ABMN)^\frac{3}8.
\\
}
This is stronger than~\eqref{bwa}, so the proof of Proposition~\ref{dsff} is concluded.
\end{proof} 

\subsection{Proof of Lemma~\ref{bfz}}
\subsubsection{Proof of~\eqref{fbfz}}
First, observe that we have
\est{
Z_{\pm,d}(x)&=\sum_{d_1d_2=d}\sum_{\substack{a,b,m_1,n_1,h\\ (a,d_2)=1,\,d|bn_1\\(am_1,bn_1/d)=1}}\sum_{0<|l|\leq L}\alpha_{d_1a}\beta_bW_0\Big(\frac {dh}H\Big) W_1\Big(\frac{d_2m_1}{M_1}\Big)W_3\Big(\frac{n_1}{N_1}\Big)\\
&\qquad\qquad W_2\Big(\frac{bn_1x}{dM_2}\Big)W_4\Big(\frac{am_1x}{N_2}\Big)  \textrm{e}\bigg(\mp lh\frac{\overline{am_1}}{bn_1/d}\bigg)\e{lx}.
}
By Weil's bound and partial summation, for $a\asymp A/d_1$ and $x\asymp \frac {dN_2}{AM_1}$, we have
\begin{eqnarray*}
&&\sum_{\substack{m_1\\(m_1,bn_1/d)=1}} W_1\Big(\frac{d_2m_1}{M_1}\Big)W_4\Big(\frac{am_1x}{N_2}\Big)   \textrm{e}\bigg(\mp lh\frac{\overline{am_1}}{bn_1/d}\bigg)\\
&&\qquad\qquad\qquad\qquad\ll_\eps (AM)^\eps(lh,bn_1/d)^\frac12(bn_1/d)^{\frac12+\eps}\Big(1+\frac{d_1M_1}{bn_1}\Big),
\end{eqnarray*}
and thus
\est{
Z_{\pm,d}(x)&\ll_\eps (AM)^\eps\sum_{d_1d_2=d}\sum_{\substack{alh\ll ALH/d\\ bn_1\ll BN_1,\, d|bn_1}}(lh,bn_1/d)^\frac12(bn_1/d)^{\frac12+\eps}\Big(1+\frac{d_1M_1}{bn_1}\Big)\\
&\ll_\eps (AM)^\eps\frac{ALH}{d^\frac32}\sum_{d_1d_2=d}\sum_{\substack{bn_1\ll BN_1\\ d|bn_1}}(bn_1)^{\frac12+\eps}\Big(1+\frac{d_1M_1}{bn_1}\Big)\\
&\ll_\eps (AM)^\eps\frac{ALH}{d^\frac52}\sum_{d_1d_2=d}\Big((BN_1)^\frac32+d_1M_1(BN_1)^\frac12\Big)\\
&\ll_\eps (AM)^\eps\frac{ALH}{d^\frac52}(BN_1)^{\frac12}\Big(BN_1+M_1\min\big\{A,H\big\}\Big),
}
since $d_1\ll A$, $d_1\leq d\ll H$.

\subsubsection{Proof of~\eqref{sbfz}}
To prove~\eqref{sbfz} we need Watt's bound in the form given by~\cite{BCR}.
\begin{lemma}\label{Watt}
Let $H,C,R,S, V,P\geq1$ and $\delta\leq1$. Assume that  
\est{%\label{cond}
X&=\Big(\frac {RSVP}{HC}\Big)^\frac12\gg (RSVP)^{\varepsilon},\qquad
(RS)^2\geq\max\Big\{H^2C,\frac{SP}{V}(RSVP)^{\varepsilon}\Big\}.
}
Moreover, assume that $\alpha(y), \beta(y)$ are complex valued smooth functions, supported on the intervals $[1,H]$ and $[1,C]$, respectively, such that
\est{
\alpha^{(j)}(x),\beta^{(j)}(x)\ll_j (\delta x)^{-j}
}
for any $j\geq0$. Assume $a_r,b_s$ are sequences of complex numbers supported on $[R,2R]$, $[S,2S]$, respectively, and such that $a_r\ll R^\varepsilon$, $b_s\ll S^\varepsilon$. Finally, assume that 
\est{
\frac{\partial^{i+j}}{\partial x^i\partial y^j}\gamma_{r,s}(x,y)\ll_{i,j} x^{-i}y^{-j}
}
for any $i,j\geq 0$, and $\gamma_{r,s}(x,y)$ is supported on $[V,2V]\times[P,2P]$ for all $r$ and $s$.
Then 
\est{%\label{oosi}
&\sum_{\substack{h,c,r,s,v,p\\(rv,sp)=1}}\alpha(h)\beta(c)\gamma_{r,s}(v,p)a_{r}b_{s} \textrm{e}\bigg(\pm\frac{hc\overline {rv}}{sp}\bigg)\\
&\ \ll \delta^{-\frac 72} HC  R(V+SX)\Big(1+\frac{HC}{RS}\Big)^\frac12\Big(1+\frac{P}{RV} \Big)^\frac12\Big(1+\frac{H^2CPX^2}{R^4S^3V}\Big)^\frac14 (HCRSVP)^{\varepsilon}.
}
\end{lemma}
In order to apply Lemma~\ref{Watt}, we write $Z_{\pm,d}(x)$ as
\est{
Z_{\pm,d}(x)=&\sum_{d_1d_2=d_3d_4=d}\sum_{\substack{a,b,m_1,n_1,h\\(am_1,bn_1)=1\\ (a,d_2)=(b,d_4)=1}}\sum_{0<|l|\leq L}\alpha_{d_1a}\beta_{d_3b}W_0\Big(\frac {dh}H\Big)W_1\Big(\frac{d_2m_1}{M_1}\Big)W_3\Big(\frac{d_4n_1}{N_1}\Big)\\
&\qquad\qquad W_2\Big(\frac{bn_1x}{M_2}\Big)W_4\Big(\frac{am_1x}{N_2}\Big) \textrm{e}\bigg(\mp lh\frac{\overline{am_1}}{bn_1}\bigg)\e{lx}.
}
Thus, we use Proposition~\ref{Watt} with
\est{
&H\leftrightarrow \frac Hd, \quad C\leftrightarrow L= \Big(\frac{ABM_1N_1}{d^2M_2N_2}\Big)^{\frac12 }(AM)^\eps\\
&R\leftrightarrow \frac A{d_1},\quad S\leftrightarrow \frac{B}{d_3},\quad V\leftrightarrow \frac{M_1}{d_2},\quad P\leftrightarrow \frac{N_1}{d_4}, 
}
and 
\[
X\leftrightarrow \Big(\frac{ABMN}{H^2}\Big)^{\frac14}%\asymp \bigg(\frac{ABMN}{H(AM+BN)}\bigg)^{1/2}
,\qquad \delta\leftrightarrow(AM)^{-\eps},
\]
since $Lx\ll (ABMN)^\eps$.  The conditions required by Lemma~\ref{Watt} are $\frac{ABMN}{H^2}\gg (ABM_1N_1)^{\varepsilon}$, which is satisfied since $H\ll (AB)^{\frac12+\eps}$ and $MN \gg (AB)^\delta$, and 
\begin{equation*}
\quad (AB)^2\geq\max\bigg\{\frac{(d_1d_3)^2H^2(ABMN)^{\frac12}}{d^3M_2N_2},\frac{d_1d_{3}^{2}BN_1}{M_1}(AM)^\varepsilon\bigg\}.
\end{equation*}
%If we assume $N_1\leq M_1$, then these conditions can be weaken as
%\begin{equation}\label{8}
%d\leq\min\bigg\{\frac{(AB)^2M_2N_2}{H^2(AM+BN)},\sqrt{AB}(AM)^{-\epsilon}\bigg\}.
%\end{equation}
%\begin{equation}\label{8}
%d\leq\min\bigg\{\frac{(AB)^\frac32}{H^2}\pr{\frac{M_2N_2}{M_1N_1}}^\frac12,\sqrt{AB}(AM)^{-\epsilon}\bigg\}.
%\end{equation}
Since $M_1\leq M_2(AM)^\varepsilon,$ $N_1\leq N_2(AM)^\varepsilon,$ $BN_1\leq AM_1$, $d_1\ll A,d_3\ll B$ and $H\ll (AB)^{\frac12+\eps}$, this condition is satisfied if $d\ll (AB)^\frac12(AM)^{-100\eps}$. Thus, under this condition we have
\est{
Z_{\pm,d}(x)&\ll (AM)^{\eps}\frac{ALH}{dd_1}  \bigg(\frac{M_1}{d_2}+\frac B{d_3}\Big(\frac{BN}{H}\Big)^{\frac12}\bigg)\bigg(1+\frac{H}{d_2d_4(AB)^\frac12}\Big(\frac{M_1N_1}{M_2N_2}\Big)^\frac12\bigg)^\frac12\\
&\qquad\qquad\qquad\qquad\qquad\qquad\Big(1+\frac{d_3N_1}{AM_1}  \Big)^\frac12\Big(1+\frac{d_1^3d_3^4N_1^2H}{d^3A^3B^2}\Big)^\frac14 \\
&\ll (AM)^{\eps}\frac{ALH}{dd_1}  \bigg(\frac{M_1}{d_2}+\frac B{d_3}\Big(\frac{BN}{H}\Big)^{\frac12} \bigg)\Big(1+\frac{d_1^3d_3^4N_1^2H}{d^3A^3B^2}\Big)^\frac14\\
&\ll (AM)^{\eps}\frac{ALH}{d}  \bigg(M_1+ B\Big(\frac{BN}{H}\Big)^{\frac12} \bigg)\Big(1+\frac{N_1^2H}{A^3B^2}\Big)^\frac14 \\
&\ll (AM)^{\eps}\frac{ABLH}{d}   \Big(\frac{BN}{H}\Big)^{\frac12}\Big(1+\frac{N_1^2H}{A^3B^2}\Big)^\frac14,\\
}
since $M_1\leq M_2(AM)^{\eps}$, $N_1\leq N_2(AM)^{\eps}$, $BN_1\leq AM_1$, $AM \asymp BN$ and $H\ll (AB)^{\frac12+\eps}$. This concludes the proof of the lemma.

\section{The quadratic divisor problem main term}%\label{qdps}

In this section we establish the quadratic divisor problem. This amounts to using Proposition~\ref{dsff} and to a careful analysis of the main term. We will first prove a rougher result and then deduce the slightly more flexible version stated in the introduction.

\begin{theorem}\label{qdp}
Let $A,B,H,X,T\geq1$ with  $\log (ABHX) \ll \log T$. Let $\alpha_a,\beta_b$ be sequences of complex numbers supported on $[A,2A]$ and $[B,2B]$, respectively, and such that $\alpha_a\ll A^\eps, \beta_b\ll B^\eps$. 
Let $f\in\mathcal C^{\infty}(\R_{\geq0}^{3})$ be such that
\es{\label{asof}
\frac{\partial ^{i+j+k}}{\partial x^i\partial y^j\partial z^k}f(x,y,z)\ll_{i,j,k} T^{\eps}(1+x)^{-i}(1+y)^{-j}(1+z)^{-k}
}
for any $i,j,k\geq0$. Moreover, assume $f(x,y,z)$ is supported on $[H,2H]$ as a function of $z$ for all $x,y$. Finally, let $K\in\mathcal C^\infty(\R_{\geq0})$ be such that $K^{(j)}(x)\ll_{j,r} T^\eps (1+x)^{-j}(1+x/X^2)^{-r}$ for any $j,r\geq0$. Then, writing
\est{%\label{dfms}
\mathcal{S} = 
\sum_{a m_1 m_2 - b n_1 n_2 = h > 0}\frac{ \alpha_a  \beta_b}{ m_1^{\alpha} m_2^{\beta} n_1^{\gamma} n_2^{\delta}} f(am_1 m_2, b n_1 n_2, h)K(m_1m_2n_1n_2), 
}
where the sum runs over positive integers $a,b,m_1,m_2,n_1,n_2$ and $h$, 
we have
\es{\label{mtf}
\mathcal{S} = \mathcal{M}_{\alpha,\beta,\gamma,\delta} + \mathcal{M}_{\beta,\alpha,\gamma,\delta}
+ \mathcal{M}_{\alpha,\beta,\delta,\gamma} + \mathcal{M}_{\beta,\alpha,\delta,\gamma} + \mathcal E,
}
where 
\begin{align*}
 \mathcal{M}_{\alpha,\beta,\gamma,\delta}=& \frac{\zeta(1 + \alpha -\beta)\zeta(1 + \gamma-\delta)}{\zeta(2 + \alpha - \beta + \gamma- \delta)}\sum_{g}
\sum_{(a,b) = 1} \frac{\alpha_{ga} {\beta_{gb}}g}{(g a )^{1 -\beta} (g b)^{1-\delta }} 
 \\ &\qquad\qquad \int_{0}^{\infty} K\Big(\frac{x^2}{g^2ab}\Big)\widetilde{f}_{ \alpha, \beta, \gamma, \delta}(x, x; a,b,g) x^{-\beta-\delta } dx,
\end{align*}
with $\widetilde{f}_{ \alpha, \beta, \gamma, \delta}$ as in~\eqref{dfntildef}, and $\mathcal E$ is bounded by
\es{\label{mb}
\mathcal E \ll  T^\eps(A B X^2 H^2)^{\frac14} \big ( AB + H^{\frac14}(A + B)^{\frac12} 
(ABX^2)^{\frac18} \big )
}
if $H \ll (A B)^{\frac12 + \varepsilon}$, and by 
\es{\label{wb}
\mathcal E&\ll  T^\eps(ABX^2H^2)^\frac38(ABH)^\frac14(A+B)^{\frac54}+T^\eps H^2
}
in any case.
\end{theorem}
\begin{proof}
First notice that we can replace the assumption~\eqref{asof} by a stronger one,
\est{%\label{asof2}
\frac{\partial ^{i+j+k}}{\partial x^i\partial y^j\partial z^k}f(x,y,z)\ll_r T^{\eps}(1+x)^{-i}(1+y)^{-j}(1+z)^{-k}(1+xy/ABX^2T^{\eps})^{-r}
}
for any $i,j,k,r\geq0$, since both $\mathcal S$ and the main terms $\mathcal M$ change by a negligible amount when multiplying $f$ by $\kappa(xy/ ABX^2T^{\eps})$, where $\kappa(x)$ is a smooth function which is identically $1$ for $x\leq1$ and decays faster than any polynomial at infinity. 

We let $g$ be a smooth function such that
$$
g(x) + g(1/x) = 1
$$
for all $x \in \mathbb{R}$ and $g(x) \ll_{r} (1 + x)^{-r}$ for any fixed $r>0$ and $x > 1$.
We also require that
$$
\widehat{g}\Big(\pm \frac{(\alpha - \beta)}{2}\Big) =\widehat{g}\Big (\pm\frac{ (\gamma - \delta)}{2}\Big ) =  0.
$$
Introducing the product
$$
\bigg ( g \Big ( \frac{m_1}{m_2} \Big ) + g \Big ( \frac{m_2}{m_1} \Big ) \bigg )
 \bigg ( g \Big ( \frac{n_1}{n_2} \Big ) + g \Big ( \frac{n_2}{n_1} \Big ) \bigg ) = 1
$$
we obtain four roughly similar terms. For simplicity we will focus on only one of them, say,
the one with $g(\frac{m_1}{m_2})g(\frac{n_1}{n_2})$. 

We apply a dyadic partition of unity to the sums over $m_1, m_2, n_1, n_2$ and $h$. Let $W$ be a smooth non-negative function supported in $[1, 2]$ such that
$$
\sum_{M} W \Big ( \frac{x}{M} \Big ) = 1,
$$
where $M$ runs over a sequence of real numbers with $\#\{M: Y^{-1}\leq M\leq Y\}\ll \log Y$. With this partition of unity, we re-write our sum as
\begin{align*}
 \mathcal S=\sum_{M_1,M_2,N_1,N_2,H'} S(M_1, M_2, N_1, N_2,H') + O_A(T^{-A}),
\end{align*}
where
\begin{align*}
&S(M_1, M_2, N_1, N_2,H') \\
&\qquad\qquad= \sum_{a m_1 m_2 - b n_1 n_2 = h>0} \frac{\alpha_a \beta_b}{m_1^{\alpha} m_2^{\beta}
n_1^{\gamma} n_2^{\delta}} f(a m_1 m_2, b n_1 n_2, h) K(m_1m_2n_1n_2)  \\ & \qquad\qquad\qquad\qquad  g \Big ( \frac{m_1}{m_2} \Big )g \Big ( \frac{n_1}{n_2} \Big )
W \Big ( \frac{h}{H'} \Big )W \Big ( \frac{m_1}{M_1} \Big ) W \Big ( \frac{m_2}{M_2} \Big )W\Big ( \frac{n_1}{N_1} \Big ) W \Big ( \frac{n_2}{N_2}
\Big ) .
\end{align*} 
Notice that we can assume $H'\asymp H$ by our assumption on $f$. %The reason for introducing a partition of unity for h is to keep track of the support of the sum also when writing $f$ as a Mellin transform.
Using the estimates for $g$ and $K$ we obtain
\begin{align*}
\mathcal S=\sum_{\substack{M_1,M_2,N_1,N_2,H'\\ M_1\leq M_2 T^\eps,\, N_1\leq N_2 T^\eps\\ M_1M_2N_1N_2\ll X^2 T^\eps}} S(M_1, M_2, N_1, N_2,H') +O_A(T^{-A}).
\end{align*}

We now separate variables in $S(M_1, M_2, N_1, N_2,H')$ by introducing the Mellin inversions,
\est{
&f(x,y, h)= \frac{1}{(2\pi i)^3} \int_{(\eps)} \int_{(\eps)} \int_{(\eps)}\widehat{f}(s,w, z) x^{-s} y^{-w} h^{-z} ds dw dz ,\\
&g(x) = \frac{1}{2\pi i} \int_{(\eps)} \widehat{g}(u) x^{-u} du\qquad\textrm{and}\qquad K(x) = \frac{1}{2\pi i} \int_{(\eps)} \widehat{K}(\nu) x^{-\nu} d\nu.
}
Note that $g(x)$ has a simple pole at $u = 0$ with residue $1$.
Thus,
\begin{align*}
& S(M_1, M_2, N_1, N_2,H') =\frac{1}{(2\pi i)^6} \int_{(\varepsilon)}  \hspace{-0.4em}\cdots \int_{(\varepsilon)}\sum_{a m_1 m_2 - b n_1 n_2 = h > 0} \frac{\alpha_a\beta_b}{a^s b^{w}}h^{-z} W\Big ( \frac{h}{H'} \Big )  \\ &\qquad  
m_1^{-s - u -\nu- \alpha}W \Big ( \frac{m_1}{M_1}
\Big )  m_2^{-s + u-\nu - \beta}W \Big ( \frac{m_2}{M_2} \Big )n_1^{-w - v-\nu - \gamma}W \Big ( \frac{n_1}{N_1} \Big ) n_2^{-w + v -\nu- \delta}W \Big ( \frac{n_2}{N_2} \Big )  \\
&\qquad\qquad\widehat{f}(s,w, z) \widehat K(\nu)  \widehat{g}(u) \widehat{g}(v)
 ds dwdz du dv d\nu.
\end{align*}

Now we apply Proposition 3.1\footnote{To be more precise, we need to truncate the complex integrals at height $\pm T^\eps$ before applying Proposition 3.1 and re-extend them afterwards, as can be done at a negligible cost thanks to~\eqref{bndhatf}.} to transform the above expression into
\begin{align*}
& \frac{1}{(2\pi i)^6} \int_{(\eps)} \ldots \int_{(\eps)}\int_{0}^{\infty} \sum_{\substack{a,b,m_1,n_1,h,d\\(a m_1, b n_1) = d  }}
 \frac{\alpha_a\beta_b}{a^s b^{w}} (hd)^{-z} m_1^{-s -u-\nu -\alpha}\Big ( \frac{b n_1 x}{d} \Big )^{-s + u-\nu - \beta}\\
&\qquad  n_1^{-w -v-\nu - \gamma}\Big ( \frac{a m_1 x}{d} \Big )^{-w +v -\nu- \delta}W\Big ( \frac{dh}{H'} \Big )W \Big ( \frac{m_1}{M_1} \Big )W \Big ( \frac{b n_1 x}{d M_2}
\Big )W \Big ( \frac{n_1}{N_1}
\Big )W \Big( \frac{a m_1 x}{d N_2} \Big )   \\ 
 &\qquad\qquad\widehat{f}(s,w, z)  \widehat K(\nu) \widehat{g}(u) \widehat{g}(v)  dx ds dwdz du dv d\nu +\mathcal E_0=\mathcal M_0+\mathcal E_0,
\end{align*}
say, where
\est{
\mathcal E_0&\ll_\eps  T^\eps(A B X^2 H^2)^{\frac14} \big ( AB + H^{\frac14}(A + B)^{\frac12} 
(ABX^2)^{\frac18} \big )\\
&\qquad\qquad\int_{(\varepsilon)}\ldots\int_{(\varepsilon)} \big| \widehat{f}(s,w, z)\widehat K(\nu)\big| | dsdwdzd\nu|   \\
&\ll T^\eps(A B X^2 H^2)^{\frac14} \big ( AB + H^{\frac14}(A + B)^{\frac12} 
(ABX^2)^{\frac18} \big )
}
if $H\ll (AB)^{\frac12+\eps}$, and 
\[
\mathcal E_0 \ll_\eps T^\eps (ABX^2H^2)^\frac38(ABH)^\frac14(A+B)^{\frac54}+T^\eps H^2
\] 
in any case, since the bounds on the derivatives of $f(x, y,z)$ give
\es{\label{bndhatf}
\widehat f(s, w,z) &\ll_{\eps,k} T^\eps X^{2\tRe(\nu)}H^{\tRe( z)} \big ((ABX^2)^{\tRe( s)}+(ABX^2)^{\tRe( w)}\big)\\
&\hspace{6em} \big((1+|s|)(1+|w|)(1+|z|)(1+\nu)\big)^{-k}
}
for $\tRe (s),\tRe (w),\tRe (z),\tRe(\nu)\geq \eps$ and any $k\geq0$, using integration by parts $k$ times with respect to each variable.

Folding back the Mellin inversions we get 
\begin{align*}
&\mathcal M_0 = \int_{0}^{\infty} \sum_{\substack{a,b,m_1,n_1,h,d\\(a m_1, b n_1) = d}}  \frac{\alpha_a\beta_b}{a^\delta b^{\beta}} W\Big ( \frac{dh}{H'} \Big ) W \Big ( \frac{m_1}{M_1} \Big )W \Big ( \frac{n_1}{N_1} \Big )
 m_1^{-\alpha - \delta} n_1^{- \beta-\gamma } d^{ \beta+\delta}  \\
&\qquad\qquad  f \Big ( \frac{a b  m_1n_1 x}{d} , \frac{a b m_1 n_1 x}{d}, dh\Big )K\Big(\frac{ab(m_1n_1x)^2}{d^2}\Big) 
g \Big ( \frac{d n_1 }{a m_1 x} \Big ) g \Big ( \frac{d m_1}{b n_1 x} \Big )\\
&\qquad\qquad\qquad\qquad  W \Big ( \frac{b n_1 x}{d M_2} \Big ) W
\Big ( \frac{a m_1 x}{d N_2} \Big ) x^{ - \beta-\delta} dx.
\end{align*}
This is summed over all $N_1,N_2,M_1,M_2$ and $H'$ satisfying $M_1\leq M_2 T^\eps$, $N_1\leq N_2 T^\eps$ and $M_1M_2N_1N_2\ll X^2 T^\eps$. These conditions can be removed at the cost of an error of size $O_A(T^{-A})$. This allows us to extend the summation over all $M_1, M_2, N_1, N_2$ and $H'$, and thus to remove the partition of unity. 
In the remaining expression we now make a linear change of the $x$ variable which gives 
\begin{align*}
\mathcal M_1&=\sum_{M_1, M_2, N_1, N_2,H'} \mathcal M_0   \\
&=\int_{0}^{\infty} \sum_{\substack{a,b,m_1,n_1,h,d\\(a m_1, b n_1) = d }} \frac{\alpha_a \beta_b}{a^{1-\beta}b^{1-\delta}} d m_1^{-1-\alpha + \beta} n_1^{-1 -\gamma + \delta}\\
&\qquad\qquad\qquad\qquad f(x,x,dh) K\Big(\frac{x^2}{ab}\Big)g \Big ( \frac{am_1^2}{x} \Big )g \Big ( \frac{bn_1^2}{x} \Big )  x^{ - \beta-\delta} dx.
\end{align*}
We now prepare for the final evaluation of $\mathcal{M}_1$ by expressing $g$ in terms of its Mellin transform and $f(x,x,dh)$ as the inverse Mellin transform of $\widehat{f}_3(x,x,s)$. Then
\begin{align*}
\mathcal M_1 &= \frac{1}{(2\pi i)^3} \int_{(\eps)}\int_{(\eps)} \int_{(1+\eps)}\int_{0}^{\infty}\widehat{f}_3(x,x,s)K\Big(\frac{x^2}{ab}\Big) \widehat{g}(u) \widehat{g}(v) x^{- \beta-\delta  + u + v}\zeta(s)   \\
&\qquad\qquad \sum_{\substack{a,b,m_1,n_1,d\\(a m_1, b n_1) = d  }} \frac{\alpha_a \alpha_b}{a^{1-\beta+u}b^{1-\delta+v}}d^{1-s}m_1^{-1 -\alpha + \beta - 2u}n_1^{-1-\gamma + \delta - 2v}    dx dsdu dv.
\end{align*}

Write $g = (a,b)$ so that $a = g a'$, $b = g b'$ and $(a', b') = 1$. 
In addition $(a' m_1, b' n_1) = d'$ with $d = g d'$. Let $k = (b', m_1)$
and $\ell = (a', n_1)$, and write $m_1 = k m_1'$ and $n_1 = \ell n_1'$. Then 
using $(a', b') = 1$, we see that $(a' m_1, b' n_1) = k \ell (a' m_1', b' n_1')$.
Thus $d = g k \ell d''$ for some $d''$. 
We re-parametrize the above sum by summing over all $g$, all  $k | b'$, $\ell | a'$ and adding the
condition that $(m_1', b'/ k) = 1$ and $(n_1', a' / \ell) = 1$. For notational simplicity
we delete the extraneous superscripts $'$ and $''$ in the resulting formula, 
\begin{align*}
\mathcal M_1 =& \frac{1}{(2\pi i)^3} \int_{(\eps)} \int_{(\eps)} \int_{(1+\eps)}\int_{0}^{\infty}
\widehat{f}_3(x,x,s) K\Big(\frac{x^2}{g^2ab}\Big) \widehat{g}(u) \widehat{g}(v)x^{- \beta-\delta  + u + v}\zeta(s) 
\\ 
&\qquad\qquad  \sum_{g} g^{-1+\beta+\delta-u-v-s}\sum_{(a,b) = 1} \frac{\alpha_{g a} \beta_{g b}}{a^{1  - \beta + u} b^{1 - \delta  + v}} \\
&\qquad\qquad\qquad\qquad\sum_{\substack{k | b \\ \ell | a}} \sum_{\substack{(m_1, n_1) = d \\ (m_1, b/k) = 1 \\ (n_1, a/ \ell) = 1}}
\frac{d^{1-s}k^{ - \alpha + \beta - 2u-s} \ell^{-\gamma+\delta-2v-s}}{m_1^{1 + \alpha - \beta + 2u}n_1^{1 + \gamma - \delta + 2v} }
 dx ds du dv.  
\end{align*}
We now let
\begin{align*}
w = \alpha - \beta + 2u\qquad\textrm{and}\qquad z =  \gamma - \delta + 2v.
\end{align*}
In this way, 
\begin{align} \label{eq:dirseries}
\sum_{\substack{(m_1, n_1) = d \\ (m_1, b/k) = 1 \\ (n_1, a/ \ell) = 1}}
\frac{d^{1-s}}{m_1^{1 + \alpha - \beta + 2u}n_1^{1 + \gamma - \delta + 2v} } & = \sum_{\substack{(m_1, n_1) = d \\ (m_1, b/k) = 1 \\ (n_1, a/ \ell) = 1}}
\frac{d^{1-s}}{m_1^{1 + w}n_1^{1 + z} }.
\end{align}
Since $(a,b) = 1$ the above Dirichlet series factors as 
$$
\prod_{\substack{p \nmid b / k \\ p \nmid a / \ell}} \bigg ( \sum_{j = 0}^{\infty} \frac{p^{j(1 - s)}}{p^{j(2+w + z)}} \sum_{\substack{m,n\geq0,\\\min \{m,n\}=0}} \frac{1}{p^{m (1+w)+n (1+z)}}  \bigg )
\prod_{p | a / \ell}\Big(1- \frac{1}{p^{1+w} }\Big)^{-1}\prod_{p | b / k}\Big(1- \frac{1}{p^{1+z} }\Big)^{-1}.
$$
The expression in the first bracket is
\begin{align*}
&\Big(1-\frac{1}{p^{1+w+z+s}}\Big)^{-1} \Big(\sum_{\substack{m,n\geq0}} \frac{1}{p^{m (1+w)+n (1+z)}}-\sum_{\substack{m,n\geq1}} \frac{1}{p^{m (1+w)+n (1+z)}}\Big)\\
&\qquad\qquad=\Big(1-\frac{1}{p^{1+w+z+s}}\Big)^{-1}\Big(1-\frac1{p^{1+z}}\Big)^{-1}\Big(1-\frac1{p^{1+w}}\Big)^{-1}\Big(1-\frac1{p^{2+w+z}}\Big)
\end{align*}
and thus~\eqref{eq:dirseries} is equal to
\begin{align*}
&\frac{\zeta(1+w+z+ s)\zeta(1+w) \zeta(1+z)}{\zeta(2+w + z)}\\
&\qquad\qquad \prod_{p | a/l}c_p(  w + z+s,  z,  w + z)\prod_{p |b/k}c_p(w + z+s, w, w + z), 
\end{align*}
where 
$$
c_p(x,y,z) = \Big(1-\frac1{p^{1+x}}\Big) \Big(1- \frac{1}{p^{1+y}} \Big)\Big(1- \frac{1}{p^{2+z}}\Big)^{-1}.
$$
%%%
%%% COMMENT
%%%
\begin{comment}
\begin{align*}
& \frac{\zeta(z + w + s - 1)\zeta(z)\zeta(w)}{\zeta(w + z)} \prod_{p | a/k} c_p^{-1}(z + w + s - 1, w, z) \cdot \frac{p^z}{p^z - 1} \times \\ & \times \prod_{p | b / \ell} c_p^{-1}(z + w + s - 1, w, z) \cdot \frac{p^w}{p^w - 1} =  \\
& = \frac{\zeta(1 + \gamma - \delta + \alpha - \beta + s + 2u + 2v)\zeta(1 + \gamma - \delta + 2u)\zeta(1 + \alpha - \beta + 2v)}{\zeta(2 + \gamma - \delta + \alpha - \beta + 2u + 2v)} \times \\
& \times \prod_{p | (a/k)(b/ \ell)} c_p^{-1}(1 + \gamma - \delta + \alpha - \beta + s + 2u + 2v, 1 + \gamma - \delta + 2u, 1 + \alpha - \beta + 2v).
\end{align*}
where
$$
c_p(y,w,z) := \frac{p^y}{p^{y} - 1} \frac{p^{z}}{p^z - 1} \frac{p^w}{p^w - 1}
\frac{p^{w + z} - 1}{p^{w + z}}.
$$
\end{comment}
%%% 
%%% END COMMENT
%%%
Combining everything together we have obtained the following formula
\begin{align*}
&\frac{1}{(2\pi i)^{3}} \int_{(\eps)} \int_{(\eps)} \int_{(1+\eps)} \int_{0}^{\infty}\widehat f_3 (x,x, s)K\Big(\frac{x^2}{g^2ab}\Big) \widehat{g}(u) \widehat{g}(v)\\
&\qquad \frac{\zeta(s)\zeta(1+ \alpha-\beta+\gamma-\delta+2u+2v+s)\zeta(1+\alpha-\beta+2u)\zeta(1+\gamma-\delta+2v) }{\zeta(2+\alpha-\beta+\gamma-\delta+2u+2v)} \\ & \qquad\qquad
x^{ - \beta-\delta + u + v}\sum_{g} g^{1+\beta+\delta-u-v-s} \sum_{(a,b) = 1} \frac{\alpha_{ga} \beta_{gb} \eta_{\alpha,\beta,\gamma,\delta,a,b}(u,v,s) }{a^{1  - \beta +u} b^{1 - \delta  +v}}   dx ds du dv,
\end{align*}
where 
\begin{equation}\label{etanew}
 \eta_{\alpha,\beta,\gamma,\delta,a,b}(u,v,s)= \eta_{\alpha,\beta,\gamma,\delta,a}(u,v,s) \eta_{\gamma,\delta,\alpha,\beta,b}(u,v,s)
\end{equation}
and
\begin{align*}
&\eta_{\alpha,\beta,\gamma,\delta,a}(u,v,s) \\
&\qquad= \sum_{\ell|a} \ell^{  - \gamma+\delta-2v-s} \\
&\qquad\qquad\prod_{p | a /\ell} c_p(  \alpha-\beta+\gamma-\delta+2u+2v+s,  \gamma-\delta+2v,  \alpha-\beta+\gamma-\delta+2u+2v).
\end{align*}

Next we shift the line integration over $u$ towards $\tRe (u) = -1/4 + \varepsilon/2$ and that of $v$ towards $\tRe( v) = -1/4 + \varepsilon/2$. We collect the poles from $u = 0$ and $v = 0$, and for the terms where only one of the two residues is taken we move the other integral to the $(-1/2+\eps)$-line so that for the three resulting error terms we always have $\tRe(u)+\tRe(v)=-1/2+\eps$. We do not collect poles at $u = - ( \alpha - \beta) / 2$ and $v = -(\gamma - \delta)/2$ since we ensured that $\widehat{g}(-(\alpha-\beta)/2) =\widehat{g}(-(\gamma - \delta)/2) = 0$. Since $\widehat f_3 (x,x, s)\ll_\eps T^\eps H$ for $\tRe(s)=1+\eps$, this operation produces an error of size $O_\eps\big(T^{\varepsilon}(ABX^2)^{\frac14 } H(A^{\frac12}+B^{\frac12})\big)$, which is acceptable for $\mathcal E$, and a main term equal to
\begin{align*}
 &\frac{\zeta(1 + \alpha - \beta)\zeta( 1+ \gamma - \delta)}{\zeta(2 + \alpha - \beta + \gamma - \delta)}  \sum_{g} \sum_{(a,b) = 1} \frac{\alpha_{ga} \beta_{gb}g }{(ga)^{1  - \beta }(g b)^{1 - \delta }}\\
&\qquad\qquad\qquad\qquad \int_{0}^{\infty}K\Big(\frac{x^2}{g^2ab}\Big)\widetilde{f}_{   \alpha,\beta,\gamma, \delta}(x, x; a,b,g) x^{- \beta-\delta } dx,
\end{align*}
where $\widetilde{f}_{ \alpha, \beta,\gamma, \delta}(x, x; a,b,g)$ is equal to 
\begin{align*}
\frac{1}{2\pi i} \int_{(1+\eps)}\widehat{f}_3(x, x, s)\zeta(s) \zeta(1 + \alpha - \beta+ \gamma - \delta  + s) g^{- s} \eta_{\alpha,\beta,\gamma,\delta,a,b}(0,0,s) 
ds,
\end{align*}as desired.
\end{proof}

\begin{corollary}%\label{mcc}
Let $A,B,X,Z,T\geq 1$ with $Z>XT^{-\eps}$ and $\log(ABXZ)\ll \log T$. Let $\alpha_a,\beta_b$ be sequences of complex numbers supported on $[1,A]$ and $[1,B]$, respectively, and such that $\alpha_a\ll A^\eps, \beta_b\ll B^\eps$. Let $f\in\mathcal C^{\infty}(\R_{\geq0}^{3})$ be such that
\est{
\frac{\partial ^{i+j+k}}{\partial x^i\partial y^j\partial z^k}f(x,y,z)\ll_{i,j,k,r} T^{\eps}(1+x)^{-i}(1+y)^{-j}(1+z)^{-k}\Big(1+\frac{z^2Z^2}{xy}\Big)^{-r}
}
for any $i,j,k,r\geq0$. Let $K\in\mathcal C^{\infty}(\R_{\geq0})$ be such that $K^{(j)}(x)\ll_r T^\eps (1+x)^{-j}(1+x/X^2)^{-r}$ for $0\leq j\leq2$ and any $r\geq0$. Then~\eqref{mtf} holds with the error term $\mathcal E$ bounded by
\begin{align*}
\mathcal E  \ll& T^{\eps} (AB)^\frac12 XZ^{-\frac12}  \Big ( AB + (A+ B)^{\frac12} (AB)^\frac14X^\frac12Z^{-\frac14} \Big ).
\end{align*}
\end{corollary}
\begin{proof}
We divide the summations over $a,b,h$ using partitions of unity localizing $a\asymp A'$, $b\asymp A'$, $h\asymp H'$ and notice that by~\eqref{wb} the error term coming from the terms with $A'B'\ll T^{ \eps}$ is bounded by $T^{\eps}  X^{\frac74}Z^{-1}\ll T^\eps X^{\frac32}Z^{-\frac34}$ . For the terms with $A'B'\gg T^{ \eps}$ we observe that the contribution from the terms with 
 $H'\gg T^\eps (A'B')^{\frac12}XZ^{-1}$ is negligible, whereas for the remaining terms we have $H'\ll T^\eps(A'B')^{\frac12}XZ^{-1}\ll (A'B')^{\frac12+\eps}$ and we can apply~\eqref{mb}. Summing back over the partitions of unity then gives the claimed result.
\end{proof}

\end{document}